\documentclass[english]{amsart}

\usepackage{amsmath} 
\usepackage{amssymb} 

\usepackage{amscd} 
\usepackage{tabularx} 
\usepackage[all,cmtip,line]{xy}
 
\def\RCS$#1: #2 ${\expandafter\def\csname RCS#1\endcsname{#2}}
\RCS$Revision: 162 $
\RCS$Date: 2015-08-18 09:06:46 +0100 (Tue, 18 Aug 2015) $

\DeclareMathOperator{\Frob}{Frob} 
   
\DeclareMathOperator{\ab}{ab}
\DeclareMathOperator{\BC}{BC}

\newcommand{\kbar}{{\overline{k}}}
\newcommand{\To}{\longrightarrow} 
\newcommand{\into}{\hookrightarrow} 
 
\newcommand{\isoto}{\stackrel{\sim}{\To}}

\newcommand{\rec}{\operatorname{rec}}

\newcommand{\calO}{\mathcal{O}}

\newcommand{\R}{\mathbb{R}} 
\newcommand{\Z}{\mathbb{Z}} 
\newcommand{\A}{\mathbb{A}} 
\newcommand{\Q}{\mathbb{Q}}

\renewcommand{\gg}{\mathfrak{g}} 
\newcommand{\gh}{\mathfrak{h}}

\newcommand{\C}{\mathbb{C}}

\newcommand{\Gal}{\operatorname{Gal}} 
\newcommand{\GL}{\operatorname{GL}}
\newcommand{\gl}{\operatorname{\mathfrak{gl}}}
\newcommand{\PGL}{\operatorname{PGL}}

\newcommand{\Qbar}{\overline{\Q}} 
\newcommand{\GQ}{\Gal(\Qbar/\Q)}
\newcommand{\Qp}{\Q_p} 
 
\newcommand{\Qpbar}{\overline{\Q}_p} 
\newcommand{\Qlbar}{\overline{\Q}_{\ell}} 
 
\newcommand{\Fbar}{\overline{F}} 
\newcommand{\Fvbar}{\overline{F}_v} 
\newcommand{\Fwbar}{\overline{F}_w} 
\newcommand{\Fv}{F_v}
\newcommand{\Fw}{F_w}

\newcommand{\Res}{\operatorname{Res}}

\newcommand{\Spec}{\operatorname{Spec}} 
\newcommand{\Ind}{\operatorname{Ind}} 
\newcommand{\SL}{\operatorname{SL}} 
 
\newcommand{\ad}{\operatorname{ad}} 
\renewcommand{\sc}{\operatorname{sc}} 
\newcommand{\tr}{\operatorname{tr}} 
 
\newcommand{\End}{\operatorname{End}} 
\newcommand{\Hom}{\operatorname{Hom}}

\newcommand{\Aut}{\operatorname{Aut}}

\newcommand{\sep}{\operatorname{sep}}

\def\smallmat#1#2#3#4{\bigl(\begin{smallmatrix}{#1}&{#2}\\{#3}&{#4}\end{smallmatrix}\bigr)}

\newcommand{\Gdual}{{\widehat{G}}}
\newcommand{\Ttilde}{{\widetilde{T}}}
\newcommand{\Gtildedual}{\widehat{\widetilde{G}}}
\newcommand{\Tdual}{{\widehat{T}}}
\newcommand{\Gm}{{\mathbb{G}_m}}
\newcommand{\LG}{{{}^{L}G}}
\newcommand{\CG}{{{}^{C}G}}

\usepackage[OT2,T1]{fontenc}
\DeclareSymbolFont{cyrletters}{OT2}{wncyr}{m}{n}
\DeclareMathSymbol{\Sha}{\mathalpha}{cyrletters}{"58}
\DeclareMathOperator{\alg}{alg}
\DeclareMathOperator{\cts}{cts}

\usepackage{amsthm}

\newtheorem{thm}{Theorem}[subsection] 
\newtheorem{corollary}[thm]{Corollary} 
\newtheorem{lemma}[thm]{Lemma} 
\newtheorem{lem}[thm]{Lemma} 
\newtheorem{prop}[thm]{Proposition} 
\newtheorem{conj}[thm]{Conjecture} \theoremstyle{definition} 
\newtheorem{defn}[thm]{Definition} \theoremstyle{remark} 
\newtheorem{rem}[thm]{Remark} \numberwithin{equation}{subsection}

\usepackage{hyperref}
\begin{document} 
\title[Automorphic Galois Representations]{The conjectural connections between automorphic representations and Galois representations} 
\author{Kevin Buzzard} \email{buzzard@ic.ac.uk}\address{Department of Mathematics, Imperial College London}

\author{Toby Gee} \email{toby.gee@imperial.ac.uk} \address{Department of Mathematics, Imperial College London}

\begin{abstract}
We state conjectures on the relationships between automorphic representations and Galois representations, and give evidence for them.
\end{abstract}
\maketitle 
\tableofcontents 
\section{Introduction.}
\subsection{}
Given an algebraic Hecke character for a number field~$F$, a classical
construction of Weil produces a compatible system of 1-dimensional
$\ell$-adic representations of $\Gal(\overline{F}/F)$.
In the late 1950s, Taniyama's work~\cite{MR0095161}
on $L$-functions of abelian varieties with complex multiplications
led him to consider certain higher-dimensional compatible systems of
Galois representations, and by the 1960s it was realised by Serre and others
that Weil's construction might well be the tip of a very large iceberg.
Serre conjectured the existence of 2-dimensional $\ell$-adic
representations of $\Gal(\Qbar/\Q)$
attached to classical modular eigenforms for the group $\GL_2$ over $\Q$,
and their existence was established by Deligne not long afterwards. Moreover,
Langlands observed that one way to attack Artin's conjecture on the
analytic continuation of Artin $L$-functions might be via first proving
that any non-trivial $n$-dimensional irreducible complex representation of the
absolute Galois group of a number field~$F$ came (in some precise sense)
from an automorphic representation for $\GL_n/F$, and then analytically
continuing the $L$-function of this automorphic representation instead.

One might ask whether one can associate ``Galois representations''
to automorphic representations for an arbitrary connected
reductive group over a number field. There are several approaches to
formalising
this problem. Firstly one could insist on working with all automorphic
representations and attempt to associate to them complex representations
of a ``Langlands group'', a group whose existence is only conjectural
but which, if it exists, should be much bigger than the absolute Galois group
of the number field (and even much bigger than the Weil group of the
number field)---a nice reference for a rigorous formulation of a conjecture
here is~\cite{arthur:note}. Alternatively one could restrict to automorphic
representations that are ``algebraic'' in some reasonable sense,
and in this case one might attempt to associate certain complex
representations of the fundamental group of some Tannakian category of motives,
a group which might either
be a pro-algebraic group scheme or a topological group. Finally,
following the original examples of Weil
and Deligne, one might again restrict to algebraic automorphic
representations, and then attempt to associate compatible systems
of $\ell$-adic Galois representations to such objects (that is,
representations of the absolute Galois
group of the number field over which the group is defined). The advantage
of the latter approach is that it is surely the most concrete.

For the group $\GL_n$ over a number field, Clozel gave a definition
of what it meant for an automorphic representation to be ``algebraic''.
The definition was, perhaps surprisingly, a non-trivial twist of a notion
which presumably had been in the air for many years. Clozel made some conjectures
predicting that algebraic automorphic representations should give
rise to $n$-dimensional $\ell$-adic Galois representations (so his
conjecture encapsulates Weil's result on Hecke characters and Deligne's
theorem too). Clozel proved some cases of his conjecture, when he
could switch to a unitary group and use algebraic geometry to construct
the representations.

The goal of this paper is to generalise (most of) the \emph{statement}
of Clozel's conjecture to
the case where $\GL_n$ is replaced by an 
arbitrary connected reductive group $G$. Let us explain the first
stumbling block in this programme. The naive conjecture would
be of the following form: if an automorphic representation
$\pi$ for $G$ is algebraic (in some reasonable sense)
then there should be a Galois representation into the $\Qpbar$-points
of the $L$-group of $G$, associated to $\pi$. But if one looks,
for example, at Proposition~3.4.4 of \cite{cht}, one sees
that they can associate $p$-adic Galois representations
to certain automorphic representations on certain compact unitary groups,
but that the Galois representations are taking values in
a group ${\mathcal G}_n$ which one can check is \emph{not} the
$L$-group of the unitary group in question (for dimension
reasons, for example). In fact there are even easier examples
of this phenomenon: if $\pi$ is the automorphic representation
for $\GL_2/\Q$ attached to an elliptic curve over the rationals,
then (if one uses the standard normalisation for $\pi$) one
sees that $\pi$ has trivial central character and hence descends
to an automorphic representation for $\PGL_2/\Q$ which one would
surely hope to be algebraic (because it is cohomological). However,
the $L$-group of $\PGL_2/\Q$ is $\SL_2$ and there is no way of
twisting the Galois representation afforded by the $p$-adic Tate module
of the curve so that it lands into $\SL_2(\Qlbar)$, because the
cyclotomic character has no square root (consider complex conjugation).
On the other hand, there do exist automorphic representation for $\PGL_2/\Q$
which have associated Galois representations into $\SL_2(\Qlbar)$;
for example one can easily build them from automorphic
representations on $\GL_2/\Q$ constructed via the Langlands-Tunnell theorem
applied to a continuous even irreducible 2-dimensional representation of
$\Gal(\Qbar/\Q)$ into $\SL_2(\C)$ with solvable image. What is going on?

Our proposed solution is the following. For a general connected
reductive group $G$, we believe that there are \emph{two} reasonable notions
of ``algebraic''. For $\GL_n$ these notions differ by a twist (and
this explains why this twist appears in Clozel's work). For some groups
the notions coincide. But for some others---for example $\PGL_2$---
the notions are disjoint. The two definitions ``differ by half the
sum of the positive roots''. We call the two notions $C$-algebraic
and $L$-algebraic. It turns out that cohomological automorphic representations
are $C$-algebraic (hence the $C$), and that given an $L$-algebraic
automorphic representation one might expect an associated Galois
representation into the $L$-group (hence the $L$).
Clozel twists $C$-algebraic representations into $L$-algebraic ones
in his paper, and hence conjectures that there should be Galois
representations attached to $C$-algebraic representations for $\GL_n$,
but this trick is not possible in general. In this paper we explicitly
conjecture the existence of $p$-adic Galois representations
associated to $L$-algebraic automorphic representations for a general
connected reductive group over a number field.

On the other hand, one must not leave $C$-algebraic representations
behind. For example, for certain unitary groups of even rank over the
rationals, all automorphic representations are $C$-algebraic and none
are $L$-algebraic at all! It would be a shame to have no conjecture at
all in these cases. We show in section~\ref{section:twisting} that
given a $C$-algebraic automorphic representation for a group $G$, it
can be lifted to a $C$-algebraic representation for a certain covering
group $\widetilde{G}$ (a canonical central extension of $G$ by $GL_1$) where
there is enough space to twist $C$-algebraic representations into
$L$-algebraic ones. After making such a twisting one would then
conjecturally expect an associated Galois representation into the
$L$-group not of $G$ but of $\widetilde{G}$. We define the $C$-group
${}^CG$ of $G$ to be ${}^L\widetilde{G}$. For example, if $\pi$ is the
automorphic representation for the group $\PGL_2/\Q$ attached to an
elliptic curve over $\Q$, we can verify that $\pi$ is $C$-algebraic,
and that $\widetilde{G}=\GL_2/\Q$ in this case, and hence one would expect
a Galois representation into $\GL_2(\Qlbar)$ associated to $\pi$,
which is given by the Tate module of the curve.  We also verify the
compatibility of the construction with that made by
Clozel--Harris--Taylor in the unitary group case.

In this paper, we explain the phenomena above in more detail.
In particular we formulate a conjecture associating $p$-adic Galois
representations to $L$-algebraic automorphic representations for an arbitrary
connected reductive group over a number field, which appears to
essentially include all known theorems
and conjectures of this form currently in the literature.
We initially imagined that such a conjecture was already
``known to the experts''.
However, our experience has been that this is not the case; in fact,
it seems that the issues that arise when comparing the definitions of
$L$-algebraic and $C$-algebraic representations were a known problem,
with no clear solution (earlier attempts to deal with this issue have
been by means of redefining the local Langlands correspondence and
the Satake isomorphism via a
twist, as in \cite{MR1729443}; however this trick only works for
certain groups). In one interesting example in the literature
where Galois representations are attached to certain cohomological
automorphic representations---the constructions of~\cite{cht}---the
Galois representations take values in a group $\mathcal{G}_n$,
whose construction seemed to us to be one whose main motivation
was that it was the group
that worked, rather than the group that came from a more conceptual
argument. We revisit this construction in~\S\ref{section:examples}.
Ultimately, we hope that this article will clarify once and for
all a variety of issues that occur when
leaving the relative safety of $\GL_n$, giving a firm framework for
further research on Galois representations into groups other than
$\GL_n$.

\subsection{Acknowledgements.}We would like to thank Jeff Adams, James Arthur,
Frank Calegari, Brian Conrad, Matthew Emerton, Wee Teck Gan, Dick Gross, Florian
Herzig, Robert Langlands, David Loeffler, Ambrus P\'al,
Richard Taylor and David Vogan for helpful discussions relating to this work. 
Particular thanks go to Gross, for giving us a copy of Deligne's 2007
letter to Serre and urging us to read it, and to
Adams and Vogan for dealing with several questions
of ours involving local Langlands at infinity which were apparently
``known to the experts'' but which we could not extract from the
literature ourselves.

The first author was supported by an EPSRC Advanced Research
Fellowship, and the second author would like to acknowledge the
support of the National Science Foundation (award number
DMS-0841491). He would also like to thank the mathematics department
of Northwestern University for its hospitality in the early stages of
this project.

\section{$L$-groups and local definitions.}\label{sec:review and notation}

In this section we give an overview of various standard facts concerning
$L$-groups, the Satake isomorphism, the archimedean local Langlands
correspondence, and basic Hodge--Tate theory,
often with a specific emphasis on certain arithmetic aspects that are not
considered relevant in many of the standard references. In summary: our
$L$-groups will be over $\Qbar$, we will keep track of the two different
$\Q$-structures in the Satake isomorphism, and our local Langlands
correspondence will concern representations of $G(\R)$
or $G(\overline{\R})$, where
$\overline{\R}$ is an algebraic closure of the reals which we do not
identify with $\C$ (note that on the other hand, all our
representations will be on $\C$-vector spaces).
This section is relatively elementary but
contains all of the crucial local definitions.

\subsection{The $L$-group.}

We briefly review the notion of an $L$-group. 
We want to view the $L$-group of a connected reductive group as a
group over $\Qbar$, rather than the more traditional $\C$, as we
shall later on be considering representations into the $\Qpbar$-points
of the $L$-group. We review the standard definitions from the
point of view that we shall be taking.

We take the approach to dual groups explained in section~1
of~\cite{MR757954}, but work over $\Qbar$. See also section~3.3
of Expos\'e XXIV of~\cite{MR0207710}, which is perhaps where the
trick of taking limits of based root data is first introduced.

Let $k$ be a field and let $G$ be a connected reductive algebraic group
over $k$. Fix once and for all a separable closure $k^{\sep}$ of $k$,
and let $\Gamma_k$ denote $\Gal(k^{\sep}/k)$.
The group~$G$ splits over $k^{\sep}$,
and if we choose a maximal torus~$T$ in~$G_{k^{\sep}}$
and a Borel subgroup~$B$ of $G_{k^{\sep}}$ containing~$T$, one can
associate the based root datum
$\Psi_0(G,B,T):=(X^*(T),\Delta^*(B),X_*(T),\Delta_*(B))$
consisting of the character and cocharacter groups of $T$, and
the roots and coroots which are simple and positive with respect to the ordering
defined by $B$. Now let $Z_G$ denote the centre of~$G$; then $(G/Z_G)(k^{\sep})$
acts on $G_{k^{\sep}}$ by conjugation, and if $T'$ and $B'$ are another choice of
maximal torus and Borel then there is an element of $(G/Z_G)(k^{\sep})$
sending $B'$ to $B$ and $T'$ to $T$, and all such elements
induce the same isomorphisms of based root data
$\Psi_0(G,B,T)\to\Psi_0(G,B',T')$.
Following Kottwitz, we
define $\Psi_0(G):=(X^*,\Delta^*,X_*,\Delta_*)$ to be the projective
limit of the $\Psi_0(G,B,T)$ via these isomorphisms. This means
in practice that given a maximal torus $T$ of $G_{k^{\sep}}$, the group $X^*$
is isomorphic to the character group of $T$ but not canonically;
however given also a Borel $B$ containing the torus, there is now
a canonical map $X^*=X^*(T)$ (and different Borels give different
canonical isomorphisms). There is a natural
group homomorphism $\mu_G:\Gamma_k\to\Aut(\Psi_0(G))$ (defined
for example in~\S1.3 of~\cite{MR546608}, or more conceptually by
transport de structure) and
if $K\subseteq k^{\sep}$ is a Galois extension of $k$ that splits $G$
then $\mu_G$ factors through $\Gal(K/k)$. 

We let $\Gdual$ denote a connected reductive group over $\Qbar$ equipped with
a given isomorphism $\Psi_0(\Gdual)=\Psi_0(G)^\vee$, the dual
root datum to $\Psi_0(G)$. There is a canonical group isomorphism
$\Aut(\Psi_0(G))=\Aut(\Psi_0(G)^\vee)$, sending an automorphism of $X^*$
to its inverse (one needs to insert this inverse to 'cancel out'
the contravariance coming from the dual construction),
and hence a canonical action of $\Gamma_k$
on $\Psi_0(G)^\vee$. If we choose a Borel, a maximal torus, and a splitting
(also called a pinning; see p.10 of~\cite{MR546587} for details and
definitions) of $\Gdual$ then, as on p.10 of~\cite{MR546587}, this
data induces a lifting $\Aut(\Psi_0(G)^\vee)\to\Aut(\Gdual)$ and hence
(via $\mu_G)$ a left action of $\Gamma_k$ on $\Gdual$. 
We define the $L$-group $\LG$ of $G$ to be the resulting semidirect product,
regarded as a group scheme over $\Qbar$ with identity component
$\Gdual$ and component group $\Gamma_k$. For $K$ a field
containing $\Qbar$ we have $\LG(K)=\Gdual(K)\rtimes\Gamma_k$.
Often in the literature people use $\LG$ to be the group that we
call $\LG(\C)$. 

Note that there is a fair amount of ``ambiguity'' in this definition.
The group $\Gdual$ is ``only defined up to inner automorphisms'',
as is the lifting of $\mu_G$. So, even if we fix our choice of $k^{\sep}$,
points in $\LG(K)$ are ``only defined up to conjugation by $\Gdual(K)$''.

Now say~$F$ is a number field and $k=F_v$ is a completion of $F$ at a
place~$v$. If $G/F$ is a connected reductive group and $G_v/F_v$
is its base extension, then we would like to be able to relate
the $L$-groups of~$G$ and $G_v$ (after choosing algebraic closures
$\Fbar$ and $\Fvbar$ of $F$ and $F_v$). To do this it we
choose a map $\Fbar\to\Fvbar$; doing this gives us a way of pulling
back our global choices of Borel and torus to the local group, and
hence gives us a way of identifying $\Psi_0(G)$ and $\Psi_0(G_v)$
and of embedding the absolute Galois group of $F_v$ into that of $F$
in a manner compatible with the actions of these groups on the based
root data. We thus get a natural map of algebraic groups
from $LG_v$ to $LG$.

If $K$ is an extension of $\Qbar$ and $\rho$ is a group homomorphism
$\Gamma_k\to\LG(K)$, then we say that $\rho$ is \emph{admissible}
if the map $\Gamma_k\to\Gamma_k$ induced by $\rho$ and
the surjection $\LG(K)\to\Gamma_k$ is the identity. We say
two admissible $\rho$s are \emph{equivalent} if they differ
by conjugation by an element of $\Gdual(K)$.

We remark that readers for whom even the choice of $k^{\sep}$ is
distasteful can avoid making this choice all together by interpreting
$\Gdual(K)$ as a scheme of groups over $\Spec(k)$ and then
interpreting equivalence classes of admissible $\rho$s as above as
elements of $H^1(\Spec(k),\Gdual(K))$.

We fix once and for all an embedding $\Qbar\to\C$. Later on, when
talking about Galois representations, we shall fix a prime number~$p$
and an embedding $\Qbar\to\Qpbar$. This will enable us to talk about
the groups $\LG(\C)$ and $\LG(\Qpbar)$.

\subsection{Satake parameters.}\label{satakeparams}

In this section, $k$ is a non-archimedean local field with
integers $\calO$, and we again fix
a separable closure $k^{\sep}$ of $k$ and set
$\Gamma_k=\Gal(k^{\sep}/k)$.
We normalise the reciprocity map $k^\times\to\Gamma_k^{\ab}$
of local class field theory so that it takes
a uniformiser to a geometric Frobenius.
We follow Tate's definitions and conventions for Weil groups---in brief,
a Weil group $W_k=W_{k^{\sep}/k}$ for $k$ comes equipped with maps
$W_k\to\Gamma_k$
and $k^\times\to W_k^{\ab}$ such that the induced map
$k^\times\to\Gamma_k^{\ab}$ is the reciprocity homomorphism of class field
theory, normalised as above.

Let $G/k$ be connected reductive group which is furthermore unramified
(that is, quasi-split, and split over an unramified extension of $k$).
Then $G(k)$ has hyperspecial maximal compact subgroups (namely
${\mathcal G}(\calO)\subseteq G(k)$, where ${\mathcal G}$
is any reductive group over $\calO$ with generic fibre~$G$); fix one, and
call it~$K$. Nothing we do will depend on this choice, but we will
occasionally need to justify this. Let $B=B_k$ be a Borel in $G$ defined
over $k$, let $T=T_k$ be a maximal torus of $B$, also defined over~$k$,
and let $T_d$ be the maximal $k$-split sub-$k$-torus of $T$.
Let $W_d$ be the subgroup of the Weyl
group of $G$ consisting of elements which map $T_d$ to itself.
Let ${}^oT$ denote the maximal compact subgroup of $T(k)$.
It follows from an easy cohomological calculation
(done for example in \S9.5 of~\cite{MR546608}) that
the inclusion $T_d\to T$ induces an isomorphism of groups
$T_d(k)/T_d(\calO)\to T(k)/{}^oT$. We normalise Haar measure on $G(k)$
so that $K$ has measure~1 (and remark that by 3.8.2 of \cite{MR546588} this normalisation
is independent of the choice of hyperspecial maximal compact~$K$).
If $R$ is a field of characteristic
zero then let $H_R(G(k),K)$ denote the Hecke algebra of bi-$K$-invariant
$R$-valued functions on $G(k)$ with compact support, and with multiplication
given by convolution. Similarly let $H_R(T(k),{}^oT)$ denote the analogous
Hecke algebra for $T(k)$ (with Haar measure normalised so that
${}^oT$ has measure~1).

The Satake isomorphism (see for example \S4.2 of~\cite{cartier:corvallis})
is a canonical isomorphism
$H_{\C}(G(k),K)=\C[X_*(T_d)]^{W_d}=H_{\C}(T_d(k),T_d(\calO))^{W_d}=H_{\C}(T(k),{}^oT)^{W_d}$,
where $X_*(T_d)$ is the cocharacter
group of $T_d$ (the Satake isomorphism is the first of these equalities;
the others are easy). We normalise the Satake isomorphism in the usual way,
so that it does not depend on the choice of the Borel subgroup containing~$T$;
this is the only canonical way to do things. This standard normalisation
is however not in general ``defined over $\Q$''---for example if $k=\Q_p$
and $G=\GL_2$ and $K=\GL_2(\Z_p)$ then the Satake isomorphism sends
the characteristic function of $K\smallmat{p}{0}{0}{1}K$ to a function
on $T(\Q_p)$ taking the value $\sqrt{p}$ on the matrix $\smallmat{p}{0}{0}{1}$.
This square root of $p$ appears because the definition of
the Satake isomorphism involves a twist by half the sum of the positive
roots of $G$ (see formula~(19) of section 4.2 of~\cite{cartier:corvallis})
and because of this twist, the isomorphism does \emph{not}
in general induce a canonical isomorphism
$H_{\Q}(G(k),K)=\Q[X_*(T_d)]^{W_d}$.

In~\cite{MR1729443} and \cite{MR1044819} this issue of square
roots is avoided by renormalising the Satake isomorphism. 
Let us stress that we shall \emph{not} do this here, and we shall think of
$H_{\Q}(G(k),K)$ and $\Q[X_*(T_d)]^{W_d}$ as giving two possibly
distinct $\Q$-structures
on the complex algebraic variety $\Spec(H_{\C}(G(k),K))$ which shall
perform two different functions---they will give
us two (typically distinct) notions of being defined over a subfield of $\C$.
We note however that if half the sum of the positive roots of $G$ is
in the weight lattice $X^*(T)$ (this occurs for example if $G$
is semi-simple and simply connected, or
a torus) then the map $\delta^{1/2}:T(k)\to\R_{>0}$ mentioned in formula~(19)
of~\cite{cartier:corvallis} is $\Q$-valued (see formula~(4)
of~\cite{cartier:corvallis} for the definition of~$\delta$)
and the proof of Theorem~4.1 of~\cite{cartier:corvallis}
makes it clear that in this case
the Satake isomorphism does induce
an isomorphism $H_{\Q}(G(k),K)=\Q[X_*(T_d)]^{W_d}$.

Next we recall how the Satake
isomorphism above leads us to an unramified local Langlands correspondence.
Say $\pi$ is a smooth admissible irreducible complex representation of
$G(k)$, and assume that $\pi^K\not=0$ for our given choice of $K$
(so in particular, $\pi$ is unramified).
Then $\pi^K$ is a 1-dimensional representation of $H_{\C}(G,K)$ and
hence gives rise to a $\C$-valued character of $H_{\C}(G,K)$.
Now results of Gantmacher and Langlands
in section~6 of~\cite{MR546608} enable one to canonically associate
to $\pi$ an unramified continuous
admissible representation $r_{\pi}:W_k\to\LG(\C)$,
where ``admissible'' in this context means
a group homomorphism such that if one composes it with the canonical
map $\LG(\C)\to\Gamma_k$ then one obtains the canonical
map $W_k\to\Gamma_k$ (part of the definition of the Weil group)
and ``unramified'' means that the resulting 1-cocycle
$W_k\to\Gdual(\C)$ is trivial on the inertia subgroup of $W_k$.
In fact there are two ways of normalising the construction:
if we follow section~6 of~\cite{MR546608} then in the crucial
Proposition~6.7 Borel has chosen the $\sigma$ that appears there
to be an \emph{arbitrary} generator
of the Galois group of a finite unramified extension of~$k$ which
splits $G$, and we are free to choose $\sigma$ to be either an
arithmetic or a geometric Frobenius. We shall let $\sigma$ denote
a geometric Frobenius: now an easy check (unravelling the
definitions in~\cite{MR546608} and~\cite{cartier:corvallis})
shows that if $G=\GL_1$
and $\pi$ is an unramified representation of $\GL_1(k)$
then the corresponding Galois
representation $W_k\to\GL_1(\C)$ is the one induced by our given
isomorphism $k^\times=W_k^{\ab}$. See Remark \ref{swapping geometric and arithmetic frobenius via an involution} for some comments about what would have happened
had we chosen an arithmetic Frobenius here, and normalised our class field
theory isomorphisms so that they sent arithmetic Frobenii to uniformisers.

\begin{rem}\label{unramlltwist}
The local Langlands correspondence behaves in a natural
way under certain unramified twists. Let $\xi:G\to\Gm$ be a character of~$G$
defined over~$k$. Then $\xi(K)$ is a compact subset of $k^\times$
and is hence contained in~$\calO^\times$. If $s\in\C$ then
we can define $\chi:G(k)\to\C^\times$ by $\chi(g)=|\xi(g)|^s$; then
$\chi(K)=1$. Hence for an irreducible $\pi$ with a $K$-fixed
vector, $\pi\otimes\chi$ is also irreducible with a $K$-fixed vector.
The map $\chi$ induces an automorphism of $H_\C(G(k),K)$ sending
a function $f$ to the function $g\mapsto \chi(g)f(g)$, and
similarly it induces automorphisms of $H_\C(T(k)/{}^oT)$ and
$H_\C(T_d(k)/T_d(\calO))$. These latter
two automorphisms are equivariant for the action of $W_d$, because $\chi$ is
constant on $G(k)$-conjugacy classes. Furthermore, the Satake isomorphism
commutes with these automorphisms---this follows without too much trouble
from the formula defining the Satake isomorphism in
\S4.2 of~\cite{cartier:corvallis} and
the observation that the kernel of $\xi$ contains the unipotent radical
of any Borel in~$G$. Now if
$Y=\Spec(H_\C(T_d(k),T_d(\calO))$ is the complex torus dual to $T_d$,
then given an irreducible $\pi$ with a $K$-fixed vector the Satake isomorphism
produces a $W_d$-orbit of elements of~$Y$. Twisting $\pi$ by $\chi$
amounts to changing this orbit by adding to it the $W_d$-stable
element $y_\chi$ of $Y$ corresponding to the character of $T_d(k)/T_d(\calO)$
induced by $\chi$. Let us define an element $t_\chi$ of
$\widehat{T}=X^*(T)\otimes\C^\times$, the complex torus dual to~$T$,
by $t_\chi=\xi\otimes q^{-s}=\xi\otimes|\varpi|^s$ (with $q$ the size of
the residue field of $k$ and $\varpi$ a uniformiser).
The inclusion $T_d\to T$ induces
a surjection $\widehat{T}\to Y$ and some elementary unravelling shows
that it sends $t_\chi$ to $y_\chi$. It follows that if $r_\pi:W_k\to{}^LG$
is the representation of $W_k$ corresponding to~$\pi$, then
$r_{\pi\otimes\chi}(w)=\hat\xi(|w|^s)r_\pi(w)$, with
$\hat\xi:\C^\times\to\widehat{T}\subseteq\widehat{G}$ the cocharacter
of $\widehat{T}$ corresponding to the character $\xi$ of~$T$.
\end{rem}

We now come back to the $\mathbf{Q}$-structures on $H_{\C}(G(k),K)$.
As we have already mentioned, we have a natural $\Q$-structure
on $H_{\C}(G(k),K)$ coming from the $\Q$-valued functions $H_{\Q}(G(k),K)$,
and we have another one coming from $\Q[X_*(T_d)]^{W_d}$ via
the Satake isomorphism. This means that, for
a smooth irreducible admissible
representation $\pi$ of $G(k)$ with a $K$-fixed vector,
there are two (typically distinct) notions of what it means
to be ``defined over $E$'', for
$E$ a subfield of~$\C$. Indeed, if $\pi$ is a smooth admissible
irreducible representation of $G(k)$ with a $K$-fixed vector,
then $\pi^K$ is a 1-dimensional complex vector space on which
$H_{\C}(G(k),K)$ acts, and this action induces maps
$$H_{\Q}(G(k),K)\to\C$$ and $$\Q[X_*(T_d)]^{W_d}\to\C.$$

\begin{defn}\label{definedoverEdef}
Let $\pi$ be smooth, irreducible and admissible,
with a $K$-fixed vector. Let $E$ be a subfield of $\C$.

(i) We say that \emph{$\pi$ is defined over $E$} if the induced
map $H_{\Q}(G(k),K)\to\C$ has image lying in $E$.

(ii) We say that \emph{the Satake parameter of $\pi$ is defined over $E$}
if the induced map $\Q[X_*(T_d)]^{W_d}\to\C$ has image lying in~$E$.
\end{defn}

If half the sum of the positive roots of $G$
is in the lattice $X^*(T)$, then these notions coincide. However
there is no reason for them to coincide in general and we shall shortly
see examples for $\GL_2(\Q_p)$ where they do not.

Note also that it is not immediately clear that these notions are independent
of the choice of $K$: perhaps there is some $\pi$
with a $K$-fixed vector and a $K'$-fixed vector for two non-conjugate
hyperspecial maximal compacts (for example, the trivial 1-dimensional
representation of $\SL_2(\Q_p)$ has this property), and which is defined
over $E$ (or has Satake parameter defined over $E$) for one
choice but not for the other. The reader should bear in mind that for the time being
these notions depend on the choice of~$K$, although we will soon
see (in Corollary~\ref{Carith_well_defined} and Lemma~\ref{Larith_well_defined})
that they are in fact independent of this choice.

We now discuss some other natural notions of being ``defined over $E$''
for $E$ a subfield of $\C$, and relate them to the notions above.
So let $E$ be a subfield of $\C$ and let $\pi$ be a smooth
irreducible admissible complex representation of $G(k)$ with a $K$-fixed
vector, for our fixed choice of~$K$.
Let $V$ be the underlying vector space for $\pi$.

\begin{lemma}\label{definedoverE} The following are equivalent:

(i) The representation $\pi$ is defined over~$E$.

(ii) There is an $E$-subspace $V_0$ of $V$ which is $G(k)$-stable and such
that $V_0\otimes_E\C=V$.

(iii) For any (possibly discontinuous) field automorphism
$\sigma$ of $\C$ which fixes $E$ pointwise, we have
$\pi\cong\pi^{\sigma}=\pi\otimes_{\C,\sigma}\C$ as $\C$-representations.
\end{lemma}
\begin{proof} That (ii) implies (iii) is clear---it's an abstract
representation-theoretic fact. Conversely if (iii) holds, then
(ii) follows from Lemma~I.1 of~\cite{Waldspurger:MR783510} (note: his~$E$
is not our~$E$), because
$V^K$ is 1-dimensional. This latter lemma of Waldspurger
also shows that if (ii) holds then $V_0^K$ is 1-dimensional over~$E$,
and hence (ii) implies~(i). To show that (i) implies (ii) we look
at the explicit construction giving $\pi$ from the algebra
homomorphism $H_{\Q}(G(k),K)\to\C$ given in~\cite{cartier:corvallis}.
Given a homomorphism  $H_{\Q}(G(k),K)\to\C$ with image landing in~$E$,
the resulting spherical function $\Gamma:G(k)\to\C$ defined
in equation~(30) of~\cite{cartier:corvallis} is also $E$-valued.
Now if we define $V_0$ to be the $E$-valued functions on $G(k)$
of the form $f(g)=\sum_{i=1}^nc_i\Gamma(gg_i)$ for $c_i\in E$
and $g_i\in G(k)$, then $G(k)$ acts on $V_0$ by right translations,
$V_0\otimes_E\C$ is the $V_\Gamma$ of \S4.4 of~\cite{cartier:corvallis},
and the arguments in~\S4.4 of \cite{cartier:corvallis} show that $\pi\cong V_0\otimes_E\C$.
\end{proof}

\begin{corollary}\label{Carith_well_defined} If $\pi$ is a smooth irreducible admissible
unramified representation of $G(k)$, then the notion of being ``defined
over $E$'' is independent of the choice of hyperspecial maximal
compact $K$ for which $\pi^K$ is non-zero.
\end{corollary}

\begin{proof} This is because condition (iii) of Lemma~\ref{definedoverE}
is independent of this choice.
\end{proof}
We now prove the analogous result for Satake parameters.
\begin{lemma}\label{Larith_well_defined} If $\pi$ is a smooth irreducible admissible
unramified representation of $G(k)$, then the notion
of $\pi$ having Satake parameter being defined over $E$
is independent of the choice of hyperspecial maximal compact~$K$
for which $\pi^K\not=0$.
\end{lemma}
\begin{proof}
Say $\pi$ is an unramified smooth irreducible
admissible representation of $G(k)$ with a $K$-fixed vector.
The Satake isomorphism associated to~$K$ gives us a character of
the algebra $H_{\C}(T(k),{}^oT)^{W_d}$ and hence a $W_d$-orbit
of complex characters of $T(k)$. Now by p45 of~\cite{MR546608}
and sections~3 and~4 of~\cite{cartier:corvallis}, $\pi$ is a subquotient
of the principal series representation attached to any one of
these characters, and Theorem~2.9 of \cite{MR0579172}
then implies that this $W_d$-orbit
of complex characters are the only characters for which $\pi$ occurs
as a subquotient of the corresponding induced representations.
Hence the $W_d$-orbit of characters, and hence
the map  $\Q[X_*(T_d)]^{W_d}\to\C$ attached to $\pi$,
does not depend on the choice of $K$ in the case when $\pi$ has
fixed vectors for more than one conjugacy class of hyperspecial
maximal compact. In particular the image of $\Q[X_*(T_d)]^{W_d}$
in $\C$ is well-defined independent of the choice of~$K$, and hence
the notion of having Satake parameter defined over $E$ is also
independent of the choice of~$K$.
\end{proof}

To clarify the meaning of having a Satake parameter defined over $E$,
we now explain that in the case of $G=\GL_n$ the notion becomes a more
familiar one. If $\pi$ is an unramified representation of $\GL_n(k)$
and we choose our Borel to be the upper triangular matrices and
our torus to be the diagonal matrices, then
the formalism above associates to $\pi$ an algebra homomorphism
$\C[X_*(T_d)]^{W_d}\to\C$. But here $T=T_d$ is the diagonal matrices,
and $W_d$ is the usual Weyl group~$W$ of~$G$.
The ring $\C[X_*(T_d)]=\C[X_*(T)]$
is then just the ring of functions on the dual torus $\Tdual$, and hence
an unramified $\pi$ gives rise to a $W_d$-orbit on $\Tdual$, which can
be interpreted as a semisimple conjugacy class $S_{\pi}$ in $\GL_n(\C)$. 

\begin{lemma} Let $G$ be the group $\GL_n/k$ and let $\pi$ be an unramified
representation of $G(k)$. Let $E$ be a subfield of $\C$. Then the following
are equivalent:

(i) The Satake parameter of $\pi$ is defined over~$E$.

(ii) The conjugacy class $S_{\pi}$ is defined over~$E$.

(iii) The conjugacy class $S_{\pi}$ contains an element of $\GL_n(E)$.
\end{lemma}
\begin{proof} The statement that the Satake parameter is defined over $E$
is precisely the statement that the induced map $\Q[X_*(T)]^W\to\C$
takes values in $E$, which is the statement that the characteristic
polynomial of an element of $S_{\pi}$ has coefficients in $E$.
Hence (i) and (ii) are equivalent. Furthermore, (ii) is equivalent
to (iii) because given a monic polynomial with coefficients in~$E$
it is easy to construct a semisimple matrix with this polynomial
as its characteristic polynomial.
\end{proof}

We leave to the reader the following elementary checks. Let $G_1$ and
$G_2$ be unramified connected reductive groups over $k$, and let $\pi_1$,
$\pi_2$ be unramified representations of $G_1(k)$, $G_2(k)$. Then
$\pi:=\pi_1\otimes\pi_2$
is an unramified representation of $(G_1\times G_2)(k)$. One can check
that $\pi$ is defined over $E$ iff $\pi_1$ and $\pi_2$ are defined over $E$,
and that $\pi$ has Satake parameter defined over $E$ iff $\pi_1$ and $\pi_2$
do. Now say $k_1/k$ is a finite unramified
extension of non-archimedean local fields,
and $G/k_1$ is unramified connected reductive,
and set $H=\Res_{k_1/k}(G)$. Then $H$ is unramified over $k_1$, and if $\pi$
is a representation of $G(k_1)=H(k)$ then $\pi$ is unramified as a
representation of $G(k_1)$ if and only if it is unramified as a representation
of $H(k)$. Furthermore, the two notions of being defined over $E$ (one for $G$
and one for $H$) coincide. Moreover, the two notions of having Satake
parameter defined over~$E$---one for~$G$ and one for~$H$---also coincide;
we give the argument for this as it is a little trickier. Let $T_d$ denote
a maximal split torus in $G$ and let $T$ denote its centralizer. The
Satake homomorphism for $G$ is an injective ring homomorphism
from an unramified Hecke algebra for $G$ into $\C[T(k_1)/U]$, with $U$ a
maximal compact subgroup of $T(k_1)$. The Satake homomorphism for $H$
is a map between the same two rings, and it can be easily checked
from the construction in Theorem~4.1
of~\cite{cartier:corvallis} that it is in fact the same map.
The homomorphism for $G$ is an isomorphism onto the subring $\C[T(k_1)/U]^{W(G)}$
of $\C[T(k_1)/U]$, with $W(G)$ the relative Weyl group for the pair $(G,T_d)$.
The homomorphism for $H$ is an isomorphism onto
$\C[T(k_1)/U]^{W(H)}$,
and hence 
$\C[T(k_1)/U]^{W(G)}=\C[T(k_1)/U]^{W(H)}$. Now intersecting
with $\Q[T(k_1)/U]$ we deduce that $\Q[T(k_1)/U]^{W(G)}=\Q[T(k_1)/U]^{W(H)}$
and hence the two $\Q$-structures---one coming from~$G$ and one
from~$H$---coincide.

We finish this section by noting that the notion of being defined over $E$ 
does not coincide with the notion of having Satake parameter
defined over $E$, if $k=\Q_p$ and $G=\GL_2$.
For example, if $\pi$ is the trivial 1-dimensional
representation of $\GL_2(\Q_p)$ then $\pi$ is defined
over $\Q$ but the Satake parameter attached
to $\pi$ has eigenvalues $\sqrt{p}$ and $1/\sqrt{p}$, so the Satake
parameter is not defined over $\Q$ (consider traces)
but only over $\Q(\sqrt{p})$. Similarly if $\pi$ is
the character $|\det|^{1/2}$ of $\GL_2(\Q_p)$ then $\pi_p$ is not defined
over $\Q$
but the Satake parameter of $\pi$ has characteristic polynomial
$(X-1)(X-p)$ and hence is defined over $\Q$. This issue of the
canonical normalisation of
the Satake isomorphism ``introducing a square root of~$p$'' is essentially
the reason that one sees two normalisations of local Langlands for $\GL_n$
in the literature---one used for local questions and one used for
local-global compatibility. We are not attempting to unify these
two notions---indeed, one of the motivations of this paper
is to draw the distinction between the two notions and explain
what each is good for.

\subsection{Local Langlands at infinity.}\label{locallanglandsatinfinity}

We recall the statements and basic properties of the local Langlands
correspondence for connected reductive groups over the real or
complex field. In fact we work in slightly more generality,
for the following reason: the groups that we will apply the
definitions and results of this section to are groups defined
over completions of number fields
at infinite places. So the following subtlety arises:
the completion of a number field
at a real infinite place is canonically isomorphic to the reals,
however the completion of a number field at a complex
place is isomorphic to the complex numbers, but not canonically.
Hence we actually work with groups defined over
either $\R$ or a degree two extension of $\R$ which will be isomorphic
to $\C$ but may or may not be canonically isomorphic to $\C$. Note however that
all our representations will be on $\C$-vector spaces -- there is
no ambiguity in our coefficient fields.

Let~$k$ be either the real numbers or an algebraic closure of the
real numbers. Let~$G$ be a connected reductive group over~$k$. Fix
an algebraic closure $\kbar$ of $k$ and let
$T\subseteq B$ be a maximal torus and a Borel subgroup of $G_{\kbar}$.
If $\pi$ is an irreducible admissible complex
representation of $G(k)$
then Langlands associates to $\pi$, in a completely canonical
way, a $\widehat{G}(\C)$-conjugacy class of
admissible homomorphisms $r=r_{\pi}$ from
the Weil group $W_k=W_{\overline{k}/k}$
of $k$ to $\LG(\C)$. Let us
fix a maximal torus~$\Tdual$ in $\widehat{G}_{\C}$ equipped
with an identification $X_*(\Tdual)=X^*(T)$.
The group $W_k$ contains
a finite index subgroup canonically isomorphic to $\kbar^\times$;
let us assume that $r(\kbar^\times)\subseteq\Tdual(\C)$ (which
can always be arranged, possibly after conjugating $r$
by an element of $\widehat{G}(\C)$). If $\sigma$
and $\tau$ denote the two $\R$-isomorphisms $\kbar\to\C$
then one checks easily that for $z\in\kbar^\times$
we have $r(z)=\sigma(z)^{\lambda_\sigma}\tau(z)^{\lambda_\tau}$ for
$\lambda_\sigma,\lambda_\tau\in X_{*}(\Tdual)\otimes\C$ such that
$\lambda_\sigma-\lambda_\tau\in X_{*}(\Tdual)$.
Note that because we may not want to fix a preferred choice of isomorphism
$\kbar=\C$, we might sometimes
``have no preference between $\lambda_\sigma$ and
$\lambda_\tau$''; this makes our presentation diverge slightly from
other standard references, where typically one isomorphism is
preferred.

Because $\Tdual(\C)$ is usually not its own normaliser in $\Gdual(\C)$,
there is usually more than one way of conjugating $r(\kbar^\times)$
into $\Tdual(\C)$, with the consequence that the pair
$(\lambda_\sigma,\lambda_\tau)\in(X_*(\Tdual)\otimes\C)^2$
is not a well-defined invariant of $r_{\pi}$; it
is only well-defined up to the (diagonal) action of the Weyl group
$W=W(G,T)$ on $(X_{*}(\Tdual)\otimes\C)^2$. For notational convenience
however we will continue to refer to the elements $\lambda_\sigma$
and $\lambda_\tau$ of $X_*(\Tdual)\otimes\C$ and will check that none
of our important later definitions depend on the choice we have made.

If $k=\R$ then recall from the construction of the $L$-group
that there is an action of $\Gamma_k$ on $X^*(T)\otimes\C$, and we can
ask how our pair $(\lambda_\sigma,\lambda_\tau)$ behaves under this
action. Note first that $\Gamma_k$ also acts on $W$, and
the $\Gamma_k$-action on $X^*(T)$ is $W$-semilinear
(i.e.\ for $\gamma\in\Gamma_k$ and $w\in W$ and $x\in X^*(T)\otimes\C$
we have
$\gamma(wx)=\gamma(w)\gamma(x)$); this follows from a careful unwinding
of the definitions of these actions. 
If~$c$ is the non-trivial element of $\Gamma_k$ then
one checks that $(c(\lambda_\tau),c(\lambda_\sigma))$ is in the same $W$-orbit
as $(\lambda_\sigma,\lambda_\tau)$.

If $k$ is isomorphic to $\C$ then $\lambda_\sigma$ and $\lambda_\tau$ are in
general unrelated, subject to their difference being in $X^*(T)$.

We can put this information together as follows. Let us go back to
the general case $k=\R$ or $k\cong\C$. 
Let $\Hom_{\cts}(\kbar,\C)=\{\sigma,\tau\}$
denote the continuous field isomorphisms from $\kbar$ to $\C$. The group
$$(X^*(T)\otimes\C)^{\Hom_{\cts}(\kbar,\C)}$$
has a diagonal action of the Weyl group~$W$. It also has an action
of $\Gal(\kbar/k)$ defined using both the action of $\Gal(\kbar/k)$
on $X^*(T)$ and the action on $\Hom_{\cts}(\kbar,\C)$. Explicitly,
if we think of an
element of $(X^*(T)\otimes\C)^{\Hom_{\cts}(\kbar,\C)}$ as a function
$F:\Hom_{\cts}(\kbar,\C)\to X^*(T)\otimes\C$, then $\gamma\in\Gal(\kbar/k)$
sends $F$ to the function sending $\alpha\in\Hom_{\cts}(\kbar,\C)$
to $\gamma F(\alpha\gamma)$. We deduce
that (for both $k=\R$ and $k\cong\C$)
$(\lambda_\sigma,\lambda_\tau)$ give us a well-defined element
of
$$\left(\left(\left(X^*(T)\otimes\C\right)^{\Hom_{\cts}(\kbar,\C)}\right)/W\right)^{\Gal(\kbar/k)},$$
and hence a well-defined element of
$$\left(\left(\left(X^*(T)\otimes\C\right)/W\right)^{\Hom_{\cts}(\kbar,\C)}\right)^{\Gal(\kbar/k)},$$

See the comments after Lemma~\ref{241} for a $p$-adic variant of this
construction.

The Weyl group orbit of $(\lambda_\sigma,\lambda_\tau)$ in
$(X^*(T)\otimes\C)^2=(X^*(T)\otimes\C)^{\Hom_{\cts}(\kbar,\C)}$ is naturally an invariant
attached to the Weil group representation $r_{\pi}$ rather than
to $\pi$ itself, but we can access a large part of it
(however, not quite all of it) more intrinsically from $\pi$
using the Harish-Chandra isomorphism. We explain the story when $k=\R$;
the analogous questions in the case $k\cong\C$ can then
be resolved by restriction of scalars.

So, for this paragraph only, we assume $k=\R$. If
we regard $G(k)$ as a real Lie group with Lie algebra $\gg$,
then our maximal torus $T$ of $G_{\kbar}$ gives rise to
a Cartan subalgebra $\gh$ of $\gg\otimes_{\R}\kbar$. If we now break
the symmetry and use $\sigma$ to identify $\kbar$ with $\C$,
we can interpret the Lie
algebra of $T\times_{\kbar,\sigma}\C$ as a complex Cartan subalgebra
$\gh_\sigma^{\C}$ of the
complex Lie algebra $\gg^\C:=\gg\otimes_{\R}\C$. We have a canonical
isomorphism $\gh_\sigma^{\C}=X_*(T)\otimes_{\Z}\C$ (this isomorphism
implicitly also uses~$\sigma$, because $X_*(T)=\Hom(\GL_1/\kbar,T)$
was computed over $\kbar$). Now via the
Harish-Chandra isomorphism (normalised in the usual way, so it is independent
of the choice of Borel) one can interpret the infinitesimal
character of $\pi$ as a $W$-orbit in
$\Hom_{\C}(\gh_\sigma^{\C},\C)=X^*(T)\otimes_\Z\C=X_*(\Tdual)\otimes_\Z\C$.
Furthermore, this $W$-orbit contains~$\lambda_\sigma$ (this seems to be
well-known; see Proposition~7.4 of~\cite{MR1216197} for a sketch
proof). On the other hand, we note that applying this to both
$\sigma$ and $\tau$ gives us a pair of $W$-orbits in $X^*(T)\otimes\C$,
whereas our original construction of $(\lambda_\sigma,\lambda_\tau)$ gives
us the $W$-orbit of a pair, which is a slightly finer piece of
information (which should not be surprising: there are reducible
principal series representations of $\GL_2(\R)$ whose irreducible
subquotients (one discrete series, one finite-dimensional) have the
same infinitesimal character but rather different associated Weil
representations).

We go back now to the general case $k=\R$ or $k\cong\C$.
We have a $W$-orbit $(\lambda_\sigma,\lambda_\tau)$ in
$(X^*(T)\otimes_{\Z}\C)^{\Hom_{\cts}(\kbar,\C)}$ attached to $\pi$.
One obvious ``algebraicity'' criterion that one could impose
on $\pi$ is that $\lambda_\sigma\in X_*(\Tdual)=X^*(T)$.
Note that $\lambda_\sigma$
is only well-defined up to an element of the Weyl group, but
the Weyl group of course preserves $X_*(\Tdual)=X^*(T)$, so the
notion is well-defined. Also $\lambda_\sigma$ depends on the isomorphism
$\sigma:\overline{k}\to\C$, but if we use $\tau$ instead then the
notion remains unchanged, because
$\lambda_\sigma-\lambda_\tau\in X_*(\Tdual)$
and hence $\lambda_\sigma\in X_*(\Tdual)$ if and only if
$\lambda_\tau\in X_*(\Tdual)$.
This notion of algebraicity is frequently used in the
literature---one can give the connected component of the
Weil group of $k$ the structure of the
real points of an algebraic group $\mathcal{S}$ over $\R$ and one is asking
here that the Weil representation associated to $\pi$ restricts
to a map $\mathcal{S}(\R)\to\LG(\C)$ induced by a morphism of
algebraic groups $\mathcal{S}_{\C}\to\LG_{\C}$ via
the inclusion $\mathcal{S}(\R)\subset\mathcal{S}(\C)$.

\begin{defn} We say that an admissible Weil group representation
$r:W_k\to{}^LG(\C)$ is \emph{$L$-algebraic} if $\lambda_\sigma\in X^*(T)$.
We say that an irreducible representation
$\pi$ of $G(k)$ is \emph{$L$-algebraic} if the Weil group
representation associated to it by Langlands is $L$-algebraic.
\end{defn}

Note that the notion of $L$-algebraicity for a Weil group representation~$r$
depends only on the restriction of $r$ to $\overline{k}^\times$, and
the notion of $L$-algebraicity for a representation of $G(k)$ depends
only on the infinitesimal character of this representation when $k=\R$
(and we shall shortly see that the same is true when $k\cong\C$).

Later on we will need the following easy lemma. Say $k=\R$ and
$(\lambda_\sigma,\lambda_\tau)$ is a representative of the $W$-orbit
on $X_*(\widehat{T})^2$ associated to an $L$-algebraic $\pi_{\infty}$.
Regard $\lambda_\sigma$ and $\lambda_\tau$ as maps $\C^\times\to\widehat{T}$.
\begin{lemma}\label{cxconj}
If $i$ is a square root of $-1$ in $\kbar$ and $j$ is the usual
element of order~4 in $W_k$ then the element
$\alpha_\infty
:=\lambda_{\sigma}(i)\lambda_\tau(i)r_{\pi_{\infty}}(j)\in\LG(\C)$
has order dividing~2,
and its $\Gdual(\C)$-conjugacy class is well-defined
independent of (a) the choice of order of $\sigma$ and $\tau$, (b)
the choice of representative $(\lambda_\sigma,\lambda_\tau)$
of the $W$-orbit and (c) the choice of square root of $-1$ in $\kbar$.
\end{lemma}
\begin{proof} Set $r:=r_{\pi}$. We have
$\lambda_\sigma(z)r(j)=r(j)\lambda_\tau(z)$, and $\lambda_\sigma(z)$
commutes with $\lambda_\tau(z')$, and from this it is easy to check
that $(\alpha_\infty)^2=1$ and that $\alpha_\infty$ is unchanged if we switch
$\sigma$ and $\tau$. Changing representative of the $W$-orbit just
amounts to conjugating $r$ by an element of $\Gdual(\C)$ and hence
conjugating $\alpha_\infty$ by this same element. Finally one checks
easily that conjugating $\alpha_{\infty}$ by $\lambda_{\sigma}(-1)$
gives us the analogous element with $i$ replaced by $-i$.
\end{proof}

The notion of $L$-algebraicity will be very important to us later,
however it is not hard to find
automorphic representations that ``appear algebraic in nature''
but whose infinite components are not $L$-algebraic.
For example one can check that if $E$ is an elliptic curve
over $\Q$ and $\pi$ is the associated automorphic representation
of $\PGL_2/\Q$, then the local component
$\pi_\infty$, when considered as a representation
of $\PGL_2(\R)$, is not $L$-algebraic: the element
$\lambda_\sigma$ above is in $X_*(\Tdual)\otimes_{\Z}\frac{1}{2}\Z$ but
not in $X_*(\Tdual)$. What has happened is that the canonical
normalisation of the Harish-Chandra homomorphism
involves (at some point in the definition)
a twist by half the sum of the positive roots,
and it is this twist that has taken us out of the lattice in the elliptic
curve example.

This observation motivates a \emph{second} notion of algebraicity---which
it turns out is the one used in Clozel's paper for the group $\GL_n$.
Let us go back to the case of a general connected reductive $G$ over $k$,
either the reals or a field isomorphic to the complexes.
Recall that we have fixed
$T\subseteq B\subseteq G_{\kbar}$ and hence we have the notion
of a positive root in $X^*(T)$. Let $\delta\in X^*(T)\otimes\C$
denote half the sum of the positive roots. We observed above
that the assertion ``$\lambda_\sigma\in X^*(T)$'' was independent of the choice
of $B$ and of the isomorphism $\overline{k}\cong\C$.
But the assertion ``$\lambda_\sigma-\delta\in X^*(T)$'' is also independent
of such choices, for if $\lambda_\sigma-\delta\in X^*(T)$ and $w$ is in the
Weyl group, then $w.\lambda_\sigma-\delta=w(\lambda_\sigma-\delta)-(\delta-w.\delta)\in X^*(T)$,
and also $\lambda_\tau-\delta=(\lambda_\sigma-\delta)+(\lambda_\tau-\lambda_\sigma)\in X^*(T)$.

\begin{defn} We say that the admissible Weil group representation
$r:W_k\to{}^LG(\C)$ is \emph{$C$-algebraic}
if $\lambda_\sigma-\delta\in X^*(T)$. We say that the irreducible
admissible representation $\pi$ of $G(k)$ is \emph{$C$-algebraic}
if the Weil group representation associated to $\pi$ via Langlands'
construction is $C$-algebraic.
\end{defn}

Again, $C$-algebraicity for $r$ only depends on the restriction of $r$ to
$\overline{k}^\times$, and $C$-algebraicity for $\pi$ only
depends on its infinitesimal character when $k=\R$ (and
as we are about to see, the same is true for $k\cong\C$).

Here are some elementary remarks about these definitions. If $\delta\in X^*(T)$
then the notions of $L$-algebraic and $C$-algebraic coincide.
If $G_1$ and $G_2$ are connected reductive over $k$, if $r_i$ ($i=1,2$)
are admissible representations $r_i:W_k\to{}^LG_i(\C)$, then there
is an obvious notion of a product $r_1\times r_2:W_k\to{}^L(G_1\times G_2)(\C)$
and $r_1\times r_2$ is $L$-algebraic (resp.\ $C$-algebraic) iff $r_1$
and $r_2$ are. One can furthermore check
that if $k$ denotes an algebraic closure of the reals
and $G/k$ is connected reductive, and if $H=\Res_{k/\R}(G)$,
and if $\pi$ is an irreducible admissible representation
of $G(k)=H(\R)$, then $\pi$ is $L$-algebraic (resp.\ $C$-algebraic)
when considered
as a representation of $G(k)$ if and only if it is $L$-algebraic
(resp.\ $C$-algebraic) when considered as a representation of $H(\R)$.
This assertion comes from a careful reading of sections~4 and~5
of~\cite{MR546608}. Indeed, if $T$ is a maximal torus of $G/k$ then
$\Res_{k/\R}(T)$ is a maximal torus of $H/\R$, and if
$\lambda_\sigma,\lambda_{\tau}\in X^*(T)\otimes\C$ are the parameters attached to a
representation of $G(k)$, then
$\lambda_\sigma\oplus\lambda_\tau$ and
$\lambda_\tau\oplus\lambda_\sigma\in (X^*(T)\oplus X^*(T))\otimes\C$
are the parameters attached to the
corresponding representation of $H(\R)$ (identifying $\widehat{H}(\C)$
with $\widehat{G}(\C)^2$), and if $\delta$ is half the
sum of the positive roots for $G$ then $\delta\oplus\delta$ is half
the sum of the positive roots for $H$. As a consequence, we see
that both $L$-algebraicity and $C$-algebraicity of a representation
$\pi$ of $G(k)$ are conditions that only depend on the
infinitesimal character of the representation of the underlying real
reductive group.

Let us again attempt to illustrate the difference between the two notions
of algebraicity by
considering the trivial 1-dimensional representation of $\GL_2(\R)$.
The Local Langlands correspondence associates to this the 2-dimensional
representation $|.|^{1/2}\oplus|.|^{-1/2}$ of the Weil group of
the reals. If we choose the diagonal torus in $\GL_2$ and identify
its character group with $\Z^2$ in the obvious way, then
we see that $\lambda_\sigma=\lambda_\tau=\delta=(\frac{1}{2},-\frac{1}{2})$. In particular,
$\lambda_\sigma$ is not in $X^*(T)$, but $\lambda_\sigma-\delta$ is,
meaning that this representation is $C$-algebraic but not $L$-algebraic.
Another example
would be the character $|\det|^{1/2}$ of $\GL_2(\R)$; this is
associated to the representation $|.|\oplus1$ of the Weil group,
and so $\lambda_\sigma=\lambda_\tau=(1,0)$ (or $(0,1)$, allowing for the Weyl group
action) and on this occasion $\lambda_\sigma$ is in $X^*(T)$ but $\lambda_\sigma-\delta$
is not, hence the representation is $L$-algebraic but not $C$-algebraic.
Finally let us consider the discrete series representation
of $\GL_2(\R)$ with trivial central character associated
to a weight~2 modular form. The associated representation of
the Weil group sends an element $z$ of $\overline{\R}^\times$
to a matrix with eigenvalues $\sqrt{z.\overline{z}}/z$
and $\sqrt{z.\overline{z}}/\overline{z}$, the square root
being the positive square root. We see that
the set $\{\lambda_\sigma,\lambda_\tau\}$ equals the set
$\{(\frac{1}{2},-\frac{1}{2}),(-\frac{1}{2},\frac{1}{2})\}$
(with ambiguities due to both the Weyl group action and the two choices
of identification of $\overline{\R}$ with $\C$)
and neither $\lambda_\sigma$ nor $\lambda_\tau$ are in $X^*(T)$, but both
of $\lambda_\sigma-\delta$ and $\lambda_\tau-\delta$ are, so again
the representation is $C$-algebraic but not $L$-algebraic.

\subsection{The Hodge--Tate cocharacter.}\label{htcc}

In this subsection, let $k$ be a finite extension of the $p$-adic
numbers $\Q_p$. Let $H$ be a (not necessarily connected) reductive
algebraic group over a fixed algebraic closure $\Qpbar$ of $\Q_p$.
Note that we do not fix an embedding $k\to\Qpbar$. Let $\kbar$ denote
an algebraic closure of $k$ and let $\rho:\Gal(\kbar/k)\to H(\Qpbar)$
denote a continuous group homomorphism. We say
that $\rho$ is crystalline/de Rham/Hodge--Tate if for some (and hence
any) faithful representation $H\to\GL_N$
over $\Qpbar$, the resulting $N$-dimensional Galois representation
is crystalline/de Rham/Hodge--Tate. Let $C$ denote the completion
of $\kbar$. Then for any continuous injection of fields $i:\Qpbar\to C$
there is an associated Hodge--Tate cocharacter $\mu_i:(\GL_1)_C\to H_C$
(where the base extension from $H$ to $H_C$ is via $i$).
We know of no precise reference for the construction of $\mu_i$ in this
generality; if $H$ were defined over $\Q_p$ and $\rho$ took
values in $H(\Q_p)$ then $\mu_i$ is constructed in~\cite{MR563476}.
The general case can be reduced to this case in the following way:
$H$ descends to group $H_0$ defined
over a finite extension $E$ of $\Q_p$, and a standard
Baire category theorem argument shows that $\rho$ takes values in $H_0(E')$
for some finite extension $E'$ of $E$. Now let $H_1=\Res_{E'/\Q_p}H_0$,
so $\rho$ takes values in $H_1(\Q_p)$, and Serre's construction of $\mu$
then yields $\mu_i$ as above which can be checked to be well-defined
independent of the choice of $H_0$ and so on via an elementary calculation
(do the case $H=\GL_n$ first). If $H^0$ is the identity component of~$H$
and one replaces $\rho$ by $m\rho m^{-1}$, with $m\in H^0(\Qpbar)$, then
$\mu_i$ becomes $m\mu_i m^{-1}$.

Note that there is a choice of sign that one has to make when defining $\mu_i$;
we follow Serre so, for example, the cyclotomic character gives
rise to the identity map $\GL_1\to\GL_1$.

The $H^0_C$-conjugacy class of $\mu_i$
arises as the base extension (via~$i$) of a cocharacter
$\nu_i:(\GL_1)_{\Qpbar}\to H^0$ over $\Qpbar$. Now any $i:\Qpbar\to C$
is an injection whose image is~$\kbar$ and hence induces an isomorphism
$j=$``$i^{-1}$''$:\kbar\to\Qpbar$. We set 
$\nu_j:=\nu_i$, an $H^0$-conjugacy class of
maps $\GL_1\to H$ over $\Qpbar$.

In applications, $H$ will be related to an $L$-group as follows.
If $G$ is connected and reductive over $k$, and
$\rho:\Gal(\kbar/k)\to{}^LG(\Qpbar)$ is an admissible representation,
then, because $G$ splits over a finite Galois extension~$k'$ of $k$, $\rho$
will descend to a representation
$\rho:\Gal(\kbar/k)\to\Gdual(\Qpbar)\rtimes\Gal(k'/k)$.
The target group can be made into the $\Qpbar$-points of an algebraic
group $H$ over $\Qpbar$, and if the associated representation
is Hodge--Tate then the preceding arguments associate a
$\Gdual(\Qpbar)=H^0(\Qpbar)$-conjugacy
class of maps $\nu_j:\GL_1\to H_{\Qpbar}$ to each $j:\kbar\to\Qpbar$. If $\Tdual$
is a torus in $\Gdual$ as usual, then $\nu_j$ gives rise to an
element of $X_*(\Tdual)/W=X^*(T)/W$, with $T$ a maximal
torus of $G_{\kbar}$ and $W$ its Weyl group.
In particular this construction as $j$ varies gives an element~$\nu$ of
$$(X^*(T)/W)^{\Hom_{\cts}(\kbar,\Qpbar)}.$$ We are very grateful to Florian Herzig
for drawing our attention to this construction and pointing out the formal
similarity with the calculations in the previous section.

Just as in the previous section, $\Gal(\kbar/k)$ acts on
$(X^*(T)/W)^{\Hom_{\cts}(\kbar,\Qpbar)}$: it acts on each $X^*(T)/W$
in the usual way, and it also
acts on $\Hom_{\cts}(\kbar,\Qpbar)$ by composition, so if $F:\Hom_{\cts}(\kbar,\Qpbar)\to X^*(T)/W$ and $\gamma\in\Gal(\kbar/k)$ we can define
$\gamma F$ by $(\gamma F)(\alpha)=\gamma(F(\alpha\gamma))$.
\begin{lemma}\label{241} $\nu$ is invariant under the action of $\Gal(\kbar/k)$.
\end{lemma}
\begin{proof} If we unravel what is being claimed, we see that it
suffices to prove the following. Say $j:\kbar\to\Qpbar$ is a continuous
isomorphism, and $\gamma\in\Gal(\kbar/k)$. We wish to show
that $\nu_{j\gamma}$ is $H^0$-conjugate to $\gamma^{-1}(\nu_j)$
which (by the construction of the $L$-group as a semidirect product)
is equal to the cocharacter sending $z\in\Qpbar^\times$
to $\rho(\gamma)^{-1}\nu_j(z)\rho(\gamma)$. The map $j\gamma:\kbar\to\Qpbar$
has an inverse $\gamma^{-1}i:\Qpbar\to\kbar$ which extends to a map
$\Qpbar\to C$. Let $\gamma^{-1}i\rho$ denote the induced map $\Gal(\kbar/k)\to H(C)$. 
If we base extend everything to $C$ via $\gamma^{-1}i$ then we see that
our task is to check that $\mu_{\gamma^{-1}i}$ is $H^0(C)$-conjugate
to the map
$z\mapsto(\gamma^{-1}i\rho)(\gamma)^{-1}\mu_i(z)(\gamma^{-1}i\rho)(\gamma)$. 
But in fact these two maps are \emph{equal}. To check this we can reduce
to the case $H=\GL_n/\Qpbar=\Aut(V_{\Qpbar})$, and the result then follows
from the commutativity of the following diagram (with $z\in C^\times=\GL_1(C)$).
\[\begin{gathered}[b]
\xymatrix{
V\otimes_{\gamma^{-1}i}C\ar[r]^{1\otimes\gamma}&V\otimes_iC\ar[rr]^{\rho(\gamma^{-1})\otimes\gamma^{-1}}&\hspace{1cm}&V\otimes_{\gamma^{-1}i}C\\
V\otimes_{\gamma^{-1}i}C\ar[u]^{\mu_i(z)}\ar[r]^{1\otimes\gamma}&V\otimes_iC\ar[u]^{\mu_i(\gamma
  z)}\ar[rr]^{\rho(\gamma^{-1})\otimes\gamma^{-1}}&\hspace{1cm}&V\otimes_{\gamma^{-1}i}C\ar[u]^{\mu_{\gamma^{-1}i}(z)}}\\[-\dp\strutbox]
\end{gathered}
\qedhere
\]
\end{proof}
In particular we have $$\nu\in\left((X^*(T)/W)^{\Hom_{\cts}(\kbar,\Qpbar)}\right)^{\Gal(\kbar/k)},$$ analogously to the archimedean case.
\section{Global definitions, and the first
  conjectures.}\label{sec:main defns, first conjecture}

\subsection{Algebraicity and arithmeticity.}

Let $G$ be a connected reductive group defined over a number
field $F$. Fix an algebraic closure $\overline{F}$ of~$F$
and form the $L$-group
$\LG=\Gdual\rtimes\Gal(\overline{F}/F)$ as in the previous section.
For each place $v$ of $F$, fix an algebraic closure $\overline{F_{v}}$
of $F_{v}$, and an embedding $\overline{F}\into\overline{F_{v}}$. 
Nothing we do depends in any degree of seriousness on these choices---changing
them will just change things ``by an inner automorphism''.

Let $\pi$ be an automorphic representation of $G$. Then we may write
$\pi=\otimes'_{v}\pi_{v}$, a restricted tensor product, where $v$ runs
over all places (finite and infinite) of~$F$. Recall that in
the previous section we defined notions of $L$-algebraic and
$C$-algebraic for certain representations of real and complex groups.
We now globalise these definitions.

\begin{defn}\label{defn:L-algebraic}
We say that $\pi$ is \emph{$L$-algebraic} if $\pi_v$ is $L$-algebraic
for all infinite places~$v$ of~$F$.
\end{defn}

\begin{defn}\label{defn: C-algebraic}We say that $\pi$ is \emph{$C$-algebraic}
if $\pi_v$ is $C$-algebraic for all infinite places~$v$ of~$F$.
\end{defn}

Note that, for $G=\GL_n$, the notion of $C$-algebraic coincides
(in the isobaric case) with Clozel's notion of algebraic used
in~\cite{MR1044819},
although for $\GL_2$ this choice of normalisation goes back to Hecke.
Note also that restriction of scalars preserves both notions: if
$K/F$ is a finite extension of number fields
and $\pi$ is an automorphic representation of $G/K$ then $\pi$ is
$L$-algebraic (resp.\ $C$-algebraic) when considered as a representation
of $G(\A_K)$ if and only if $\pi$ is $L$-algebraic (resp.\ $C$-algebraic)
when considered as a representation of $\Res_{K/F}(G)(\A_F)$. Indeed,
this is a local statement and we indicated the proof
in~\S\ref{locallanglandsatinfinity}.

As examples of these notions, we observe that for Hecke characters of number
fields, our notions of $L$-algebraic and $C$-algebraic both
coincide with the classical notion of being algebraic or of type $A_0$.
For $\GL_2$ the notions diverge:
the trivial 1-dimensional representation of $\GL_2(\A_{\Q})$
is $C$-algebraic but not $L$-algebraic, whereas
the representation $|\det|^{1/2}$ of $\GL_2(\A_{\Q})$
is $L$-algebraic but not $C$-algebraic. For $\GL_3/\Q$ the notions
of $L$-algebraic and $C$-algebraic coincide again (because half
the sum of the positive roots in the weight lattice); indeed they
coincide for $\GL_n$ over a number field if $n$ is odd, and differ
by a non-trivial twist if $n$ is even.

The above definitions depend only on the behaviour of $\pi$
at infinite places. The ones below depend only on $\pi$ at
the finite places; we remind the reader that the crucial
local definitions are given in Definition~\ref{definedoverEdef}.

\begin{defn}\label{defn:L-arithmetic}
We say that $\pi$ is \emph{$L$-arithmetic} if there is a finite
subset $S$ of the places of $F$, containing all infinite places and
all places where $\pi$ is ramified, and a number
field $E\subset\C$, such that for each $v\notin S$, the Satake
parameter of $\pi_v$ is defined over~$E$.
\end{defn}

\begin{defn}\label{defn:C-arithmetic}
	We say that $\pi$ is \emph{$C$-arithmetic} if there is a finite
subset $S$ of the places of $F$, containing all infinite places and
all places where $\pi$ is ramified, and a number
field $E\subset\C$, such that $\pi_v$ is defined over $E$ for
all $v\notin S$.
\end{defn}

Again we note that for $K/F$ a finite extension and $\pi$ an automorphic
representation of $G/K$, $\pi$ is $L$-arithmetic (resp.\ $C$-arithmetic)
if and only if $\pi$ considered as an automorphic representation
of $\Res_{K/F}(G)$ is.

Let us consider some examples.
An automorphic representation $\pi$ of $\GL_n/F$
will be $L$-arithmetic if there is a number field such
that all but finitely many of the Satake parameters attached to $\pi$
have characteristic polynomials with coefficients in that number
field. So, for example, the trivial 1-dimensional representation
of $\GL_2(\A_{\Q})$ would not be $L$-arithmetic, because the trace of the
Satake parameter at a prime $p$ is $p^{1/2}+p^{-1/2}=\frac{p+1}{\sqrt{p}}$,
and any subfield of $\C$ containing $(p+1)/\sqrt{p}$ for infinitely many
primes $p$ would also contain $\sqrt{p}$ for infinitely many primes $p$
and hence cannot be a number field. However
it would be $C$-arithmetic, because for all primes $p$,
$\pi_p$ is the base extension to $\C$
of a representation of $\GL_2(\A_{\Q})$ on a vector space over $\Q$.
Similarly, the representation $|\det|^{1/2}$ of $\GL_2(\A_{\Q})$
is $L$-arithmetic,
because all Satake parameters are defined over $\Q$.
However this representation
is not $C$-arithmetic: each individual $\pi_p$ is defined over a number field
but there is no number field over which infinitely many of the $\pi_p$
are defined,
again because such a number field would have to contain the square root
of infinitely many primes.

Now let $\pi$ be an arbitrary automorphic representation for an arbitrary
connected reductive group $G$ over a number field.

\begin{conj}\label{conj:LarithmeticLalgebraic}$\pi$ is $L$-arithmetic if and only if it is $L$-algebraic.\end{conj}
	
\begin{conj}\label{conj:CarithmeticCalgebraic}$\pi$ is $C$-arithmetic if and only if it is $C$-algebraic.\end{conj}

These conjectures are seemingly completely out of reach. The general
ideas behind them (although perhaps not the precise definitions
we have given) seem to be part of the folklore nowadays, although it
is worth pointing out that as far as we know the first person to raise
such conjectures explicitly was Clozel in~\cite{MR1044819}
in the case $G=\GL_n$.

For the group
$\GL_1$ over a number field both of the conjectures are true;
indeed in this case both
conjectures say the same thing, the ``algebraic implies arithmetic''
direction being relatively standard, and the ``arithmetic implies
algebraic'' direction being a non-trivial result in transcendence
theory proved by Waldschmidt in~\cite{MR608530}. We prove both conjectures
for a general torus in section~\ref{section:tori}, for the most part
by reducing to the case of $\GL_1$. On the other hand, neither
direction of either conjecture is known for the group $\GL_2/\Q$,
although in this case the conjectures turn out to be equivalent
and there are some partial results in both directions. In
particular, Sarnak has shown (\cite{MR1975448}) that for a Maass form with
coefficients in $\Z$, the associated $L$-arithmetic automorphic
representation is necessarily $L$-algebraic, and this result was generalised to the case of
coefficients in certain quadratic fields in \cite{MR1962012}. Furthermore,
if $\pi$ is a cuspidal automorphic representation for $\GL_2/\Q$
which is discrete series at infinity, then we show in~\S\ref{gl2ex}
that both conjectures hold for $\pi$. However if $\pi$ is principal
series at infinity then both directions of both conjectures are in
general open.

If one makes conjectures~\ref{conj:LarithmeticLalgebraic}
and~\ref{conj:CarithmeticCalgebraic} for all groups $G$
simultaneously, then they are in fact equivalent, by the results of
section \ref{section:twisting} below; for groups with a twisting
element (see section \ref{section:twisting} for this terminology) this follows from Propositions \ref{prop:twistinglalgtocalg}
and \ref{prop:twistinglarithtocarith}, and the general case reduces to
this one by passage to the covering groups of
Section~\ref{section:twisting}---see Proposition~\ref{conjs_are_equiv}.

\subsection{Galois representations attached to automorphic representations.}

We now fix a prime number~$p$
and turn to the notion of associating $p$-adic Galois representations
to automorphic representations. Because automorphic representations
are objects defined over $\C$ and $p$-adic Galois representations
are defined over $p$-adic fields, we need a method of passing from
one field to the other. We have already fixed an injection $\Qbar\to\C$;
now we fix once and for all a choice
of algebraic closure $\Qpbar$ of $\Qp$ and,
reluctantly, an isomorphism $\iota:\C\to\Qpbar$
of ``coefficient fields''. Recall that our $L$-groups are defined
over our fixed algebraic closure $\Qbar$ of $\Q$; our fixed inclusion
$\Qbar\to\C$ then induces, via $\iota$, an embedding $\Qbar\to\Qpbar$.
Ideally we should only be fixing an embedding $\Qbar\to\Qpbar$, and
all our constructions should only depend on the restriction of $\iota$
to $\Qbar$, but of course we cannot prove this.
Our choice of $\iota$ does affect matters, in the usual way:
for example, if $f=\sum a_nq^n$ is one of the holomorphic
cuspidal newforms for $\GL_2/\Q$ of level~1 and weight~24 then~13
splits into two prime ideals in the coefficient field of $f$,
and $a_{13}$ is in one of these prime ideals but not the other;
hence $f$ will be ordinary with respect to some choices of $\iota$
but not for others. For notational simplicity we drop $\iota$ from
our notation but our conjectural association of $p$-adic Galois
representations attached to automorphic representations will depend
very much on this choice.

We now state two conjectures on the existence of Galois
representations attached to $L$-algebraic automorphic representations,
the second stronger than the first (in that in specifies a
precise set of places at which the Galois representation is
unramified/crystalline---this is the only difference between the two
conjectures). The first version is
the more useful one when formulating conjectures about functoriality.
Note that both conjectures depend implicitly on our choice of
isomorphism $\iota:\C\to\Qpbar$ which we use to translate complex
parameters to $p$-adic ones. 
\begin{conj}\label{conj:existence of Galois representations}
If $\pi$ is $L$-algebraic, then there is a finite
subset $S$ of the places of $F$, containing all infinite places, all places dividing $p$, and
all places where $\pi$ is ramified, and a continuous Galois representation
$\rho_\pi=\rho_{\pi,\iota}:\Gal(\overline{F}/F)\to\LG(\Qpbar)$, which satisfies
\begin{itemize}
	\item  The composite of $\rho_\pi$ and the natural
          projection $\LG(\Qpbar)\to \Gal(\overline{F}/F)$ is the
          identity map.
	\item If $v\notin S$, then $\rho_\pi|_{W_{F_{v}}}$ is $\Gdual(\Qpbar)$-conjugate to $\iota(r_{\pi_{v}})$.
	\item If $v$ is a finite place dividing $p$ then
          $\rho_\pi|_{\Gal(\overline{F_v}/F_v)}$ is de
          Rham, and the Hodge--Tate cocharacter of this representation
can be explicitly read off from $\pi$ via the recipe in Remark \ref{rem: recipe for HT cocharacter}.
          \item If $v$ is a real place, let $c_{v}\in G_{F}$ denote a complex conjugation at $v$. Then $\rho_{\pi,\iota}(c_{v})$ is $\Gdual(\Qpbar)$-conjugate to the element $\iota(\alpha_v)=\iota(\lambda_{\sigma_v}(i)\lambda_{\tau_v}(i)r_{\pi_v}(j))$
of Lemma~\ref{cxconj}.
\end{itemize}	
\end{conj}
\begin{conj}\label{conj:existence of Galois representations - strong
    version with crystalline etc}
Assume that $\pi$ is $L$-algebraic. Let $S$ be the set of the places of $F$ consisting of all infinite places, all places dividing $p$, and
all places where $\pi$ is ramified. Then there is a continuous Galois representation $\rho_{\pi,\iota}:\Gal(\overline{F}/F)\to\LG(\Qpbar)$, which satisfies
\begin{itemize}
	\item  The composite of $\rho_{\pi,\iota}$ and the natural
          projection $\LG(\Qpbar)\to \Gal(\overline{F}/F)$ is the
          identity map.
	\item If $v\notin S$, then $\rho_{\pi,\iota}|_{W_{F_{v}}}$ is $\Gdual(\Qpbar)$-conjugate to $\iota(r_{\pi_{v}})$.
	\item If $v$ is a finite place dividing $p$ then
          $\rho_{\pi,\iota}|_{\Gal(\overline{F_v}/F_v)}$ is de
          Rham, and the Hodge--Tate cocharacter associated to this
representation is given by the recipe in  Remark \ref{rem: recipe for HT cocharacter}.
Furthermore, if $\pi_{v}$ is unramified then
          $\rho_{\pi,\iota}|_{\Gal(\overline{F_v}/F_v)}$ is
          crystalline.
	\item If $v$ is a real place, let $c_{v}\in G_{F}$ denote a
          complex conjugation at $v$. Then $\rho_{\pi,\iota}(c_{v})$
          is $\Gdual(\Qpbar)$-conjugate to the element $\iota(\alpha_v)=\iota(\lambda_{\sigma_v}(i)\lambda_{\tau_v}(i)r_{\pi_v}(j)))$ of
Lemma~\ref{cxconj}.
\end{itemize}	
\end{conj}
\begin{rem}\label{rem: recipe for HT cocharacter} The recipe for the Hodge--Tate cocharacter in the conjectures
above is as follows.
Say $j:\Fbar\to\Qbar$ is an isomorphism of fields. We have fixed $\Qbar\to\C$
and (via $\iota$) $\Qbar\to\Qpbar$, so $j$ induces $\Fbar\to\Qpbar$
and hence gives us a place $v|p$, an algebraic closure $\Fvbar$
of $\Fv$ equipped with an identification $j:\Fvbar\to\Qpbar$,
and a map $\Fbar\to\Fvbar$.
Similarly $j$ induces $\Fbar\to\C$ and hence
gives us an infinite place $w$ of $F$, an algebraic closure $\Fwbar$
of $\Fw$ equipped with an identification $\sigma:\Fwbar\to\C$,
and a map $\Fbar\to\Fwbar$.
Now, attached to $\pi_w$ and $\sigma$
we have constructed an element $\lambda_\sigma\in X^*(T)/W$. We need
to be careful now (indeed we thank Florian Herzig for pointing out
that we were not careful enough in the published version of this
paper) -- $T$ is a torus of $G$ defined over $\Fwbar$. Using
$\Fbar\to\Fwbar$ we can descend $\lambda_\sigma$ to an element
of $X^*(T_{\Fbar})/W$, and then we can push it out via $\Fbar\to\Fvbar$
to an element of $X^*(T_{\Fvbar})/W$. Our conjecture
is that this element $\lambda_\sigma$ is the Hodge--Tate cocharacter $\nu_j$
associated to the embedding $\Fvbar\to\Qpbar$.
\end{rem}
\begin{rem}
	The representation $\rho_{\pi,\iota}$ is not necessarily unique up to $\Gdual(\Qpbar)$-conjugation. One rather artificial reason for this is that if
$\pi$ is a non-isobaric $L$-algebraic automorphic representation of $\GL_2/\Q$
such that $\pi_v$ is 1-dimensional for almost all $v$, then there
are often many non-semisimple 2-dimensional Galois
representations that one can associate to $\pi$ (as well as a semisimple
one). But there are other
more subtle reasons too. For example if $G$ is a torus over $F$ then the
admissible Galois representations into the $L$-group of $G$ are parametrised by
$H^1(F,\widehat{G})$ (with the Galois group acting on $\widehat{G}$ via
the action used to form the $L$-group), and there may be non-zero elements
of this group which restrict to zero in $H^1(F_v,\widehat{G})$ for all places~$v$
of~$F$. If this happens then there is more than one Galois
representation that can be associated to the trivial 1-dimensional
automorphic representation of $G$. We are grateful to Hendrik Lenstra
and Bart de Smit for showing us the following explicit example of a rank~3 torus
over $\Q$ where this phenomenon occurs. If $\Gamma$ is the group
$(\Z/2\Z)^2$ and $Q$ is the quaternion group of order~8 then $Q$
gives a non-zero element of $H^2(\Gamma,\pm1)$ whose image in
$H^2(\Gamma,\C^\times)$ is non-zero but whose restriction to $H^2(D,\C^\times)$
is zero for any cyclic subgroup $D$ of $\Gamma$ (consider the corresponding
extension of $\Gamma$ by $\C^\times$ to see these facts). We now
``dimension shift''. The rank four
torus $\Z[\Gamma]\otimes_{\Z}\C^\times$ has no cohomology in degree
greater than zero and has $\C^\times$ as a subgroup, so the quotient
group $T$ is a complex torus with an action of $\Gamma$ and with
the property that there's an element of $H^1(\Gamma,T)$ whose restriction
to any cyclic subgroup is zero. Finally, $\Gamma$ is isomorphic
to $\Gal(\Q(\sqrt{13},\sqrt{17})/\Q)$ (a non-cyclic group all of whose
decomposition groups are cyclic) and $T$ with its Galois action
can be realised as the complex points of the dual group of a torus
over $\Q$, giving us our example: there is more than one Galois
representation associated to the trivial 1-dimensional representation
of this torus.
\end{rem}

\begin{rem}\label{swapping geometric and arithmetic frobenius via an
    involution}We have normalised local class field theory so that
  geometric Frobenius elements correspond to uniformisers, and defined
our Weil groups accordingly. Had we normalised things the other way
(associating arithmetic Frobenius to uniformisers) then the natural
thing to do when formulating our unramified local Langlands dictionary
would have been to use an arithmetic Frobenius as a generator of our
Galois group. In particular our unramified
local Langlands dictionary at good finite places would
be changed by a non-trivial involution. Had we made this choice initially,
conjectures~\ref{conj:existence of Galois
    representations} and~\ref{conj:existence of Galois representations
    - strong version with crystalline etc} need to be modified:
one needs to change the Hodge--Tate cocharacter
$\mu$ to $-\mu$. However these new conjectures are equivalent to the
 conjectures as stated,
  because the required Galois representation predicted by the new conjecture
may be obtained directly
  from $\rho_{\pi}$ by applying the Chevalley involution of $\LG$ (the
Chevalley involution of $\Gdual$ extends to $\LG$ and induces the identity
map on the Galois group), or indirectly as $\rho_{\tilde{\pi}}$ where
$\tilde{\pi}$ is
  the contragredient of $\pi$. We omit the formal proof that these
  constructions do the job---in fact, although the arguments at the
finite places are not hard, we confess that we were not able
to find a precise published reference for the statements at infinity that we
need. The point is that we need to know how the local Langlands correspondence
for real and complex reductive groups behaves under taking contragredients. The
involution on the $\pi$ side induced by contragredient
corresponds on the Galois side to the involution
on the local Weil representations induced by an involution of the Weil group
of the reals/complexes sending $z\in\kbar^\times\cong\C^\times$ to $z^{-1}$.
It also corresponds to the involution on the Weil representations
induced by the Chevalley involution. Both these facts seem to be well-known
to the experts but the proof seems not to be in the literature.
Jeff Adams and David Vogan inform us that they are working on
a manuscript called ``The Contragredient'' in which these issues
will be addressed.
\end{rem}

\subsection{Example: the groups $\GL_2/\Q$ and $\PGL_2/\Q$}\label{gl2ex}

The following example illustrates the differences between the $C$-\ and
$L$-\ notions in two situations, one where things can be ``fixed by twisting''
and one where they cannot. The proofs of the assertions made here
only involve standard unravelling of definitions and we shall omit
them.

Let $\A$ denote the adeles of $\Q$. For $N$ a positive
integer, let $K_0(N)$ denote the subgroup of $\GL_2(\widehat{\Z})$
consisting of matrices which are upper triangular modulo~$N$.
If $\GL_2^+(\R)$ denotes the matrices in $\GL_2(\R)$ with positive
determinant then $\GL_2(\A)=\GL_2(\Q)K_0(N)\GL_2^+(\R)$.
Now let $f$ be a modular form of weight $k\geq2$
which is a normalised cuspidal eigenform for the subgroup $\Gamma_0(N)$
of $\SL_2(\Z)$, and let $s$ denote a complex number. We think
of $f$ as a function on the upper half plane.
Associated to $f$ and $s$ we define a function $\phi_s$ on $\GL_2(\A)$
by writing an element of $\GL_2(\A)$ as $\gamma\kappa u$
with $\gamma\in\GL_2(\Q)$, $\kappa\in K_0(N)$ and
$u=\smallmat{a}{b}{c}{d}\in\GL_2^+(\R)$, and defining
$\phi_s(\gamma\kappa u)=(\det u)^{k-1+s}(ci+d)^{-k}f((ai+b)/(ci+d))$.
This function is well-defined and is a cuspidal automorphic form,
which generates an automorphic representation $\pi_s$ of $\GL_2(\A)$.
The element $s$ is just a twisting factor; if $s$ is a generic
complex number then $\pi_s$ will not be algebraic or arithmetic
for either of the ``$C$'' or ``$L$'' possibilities above. 

First we consider the arithmetic side of the story.
For $p$ a prime not dividing $N$, let $a_p$ is the coefficient
of $q^p$ in the $q$-expansion of $f$. It is well-known
that the subfield of $\C$ generated by the $a_p$ is a number field $E$.
An elementary but long explicit calculation shows the following.
If $\pi_{s,p}$ denotes the local component of $\pi_s$ at $p$,
then $\pi_{s,p}$ has a non-zero invariant vector under the
group $\GL_2(\Z_p)$ and the action of the Hecke operators $T_p$
and $S_p$ on this 1-dimensional space are via the complex
numbers $p^{2-k-s}a_p$ and $p^{2-k-2s}$. The Satake parameter
associated to $\pi_{s,p}$ is the semisimple conjugacy class
of $\GL_2(\C)$ consisting of the semisimple elements with
characteristic polynomial $X^2-a_pp^{3/2-k-s}X+p^{2-k-2s}$.
Hence $\pi_s$ is $L$-arithmetic if $s\in \frac{1}{2}+\Z$.
In fact one can go further.
By the six exponentials theorem of transcendental number theory
one sees easily that a complex number $c$ with the property
that $p^c$ is algebraic for at least three prime numbers~$p$
must be rational. Hence if $\pi_s$ is $L$-arithmetic
then $s$ is rational and (because a number field is only ramified
at finitely many primes and hence cannot contain
the $t$th root of infinitely many prime numbers for any $t>1$)
one can furthermore deduce that $2s\in\Z$. Next one observes
that $a_p$ must be non-zero for infinitely many primes~$p\nmid N$
(because one can apply the Cebotarev density theorem
to the mod $\ell>2$ Galois representation associated
to~$f$ and to the identity matrix) and deduce (again because a number
field cannot contain the square root of infinitely many primes)
that $\pi_s$ is $L$-arithmetic iff $s\in\frac{1}{2}+\Z$. Now
by Proposition~\ref{prop:twistinglarithtocarith}
(whose proof uses nothing that we haven't
already established) we see that $\pi_s$ is $C$-arithmetic iff $s\in\Z$.

We now consider the algebraic side of things.
If $\pi_{s,\infty}$ denotes the local component of $\pi_s$ at
infinity and we choose the Cartan subalgebra $\gh^{\C}$ of $\gl_2(\C)$
spanned by $H:=\smallmat{1}{0}{0}{-1}$ and $Z:=\smallmat{1}{0}{0}{1}$
then the infinitesimal character of $\pi_{s,\infty}$ (thought of
as a Weil group orbit in $\Hom_{\C}(\gh^{\C},\C)$) sends $H$ to $\pm(k-1)$
and $Z$ to $2s+k-2$. The characters of the torus in $\GL_2(\R)$
give rise to the lattice $X^*(T)$ in $\Hom(\gh^{\C},\C)$ consisting 
of linear maps that send $H$ and $Z$ to integers of the same parity.
Hence $\pi_s$ is $C$-algebraic iff $s\in\Z$ and $L$-algebraic iff
$s\in \frac{1}{2}+\Z$. In particular $\pi_s$ is $L$-algebraic iff
it is $L$-arithmetic, and $\pi_s$ is $C$-algebraic iff it is $C$-arithmetic.

We now play the same game for Maass forms. If $f$ (a real analytic
function on the upper half plane) is a cuspidal Maass form
of level $\Gamma_0(N)$ which is an eigenform for the Hecke operators,
and $s\in\C$ then one can define a function $\phi_s$ on $\GL_2(\A)$ by
writing an element of $\GL_2(\A)$ as $\gamma\kappa u$ as above,
writing $u=\smallmat{a}{b}{c}{d}$, 
and defining $\phi_s(\gamma\kappa u)=\det(u)^s f((ai+b)/(ci+d))$.
If we now assume that $f$ is the Maass form associated by Langlands and
Tunnell to a Galois representation $\rho:\GQ\to\SL_2(\C)$ with solvable image,
and if $p\nmid N$ is prime and $a_p=\tr(\rho(\Frob_p))$, then the
$a_p$ generate a number field $E$ (an abelian extension of $\Q$ in this
case) and a similar explicit calculation, which again we omit, shows
that $\pi_s$ is $L$-arithmetic iff $\pi_s$ is $L$-algebraic
iff $s\in\Z$, and that $\pi_s$ is
$C$-arithmetic iff $\pi_s$ is $C$-algebraic
iff $s\in\frac{1}{2}+\Z$.

Note in particular that the answer in the Maass form
case is different to the holomorphic case in the sense
that $s\in\Z$ corresponded to the $C$-side in the holomorphic
case and the $L$-side in the Maass form case.

An automorphic representation for $\PGL_2/\Q$ is just an automorphic
representation for $\GL_2/\Q$ with trivial central character. One checks
that the $\pi_s$ corresponding to the holomorphic modular form has trivial
central character iff $s=1-\frac{k}{2}$ (this is because the form
was assumed to have trivial Dirichlet character) and, again because
the form has trivial character, $k$ must be even so in particular the
$\pi_s$ which descends to $\PGL_2/\Q$ is $C$-algebraic and $C$-arithmetic.
However, the $\pi_s$ corresponding to the Maass form with trivial
character has trivial central character iff $s=0$, which is $L$-algebraic
and $L$-arithmetic. Hence, when applied to the group $\PGL_2/\Q$,
our conjecture above says that there should be a Galois representation
to $\SL_2(\Qlbar)$ associated to the Maass form but it says nothing
about the holomorphic form. However, the holomorphic form is clearly
algebraic in some sense and indeed there is a Galois representation
associated to the holomorphic form---namely the Tate module of the
elliptic curve. Note however that the determinant of the Tate module
of an elliptic curve is the cyclotomic character, which is not the
square of any 1-dimensional Galois representation (complex conjugation
would have to map to an element of order~4) and hence no twist of
the Tate module of an elliptic curve can take values in $\SL_2(\Qlbar)$.
This explains why we have thus far restricted
to $L$-algebraic representations for our
conjecture attaching Galois representations to automorphic representations.

Finally, we note that it is easy to check that the automorphic
representations corresponding to Hilbert modular forms with
weights that are \emph{not} all congruent modulo~2, are neither
$C$-algebraic nor $L$-algebraic
(cf. pp.91--92 of \cite{MR1044819}).

\subsection{Why $C$-algebraic?}

Our conjecture above only attempts to associate Galois representations
to $L$-algebraic automorphic representations. So why consider
$C$-algebraic representations at all? For $\GL_n$ the issue is only
one of twisting: $\pi$ is $L$-algebraic iff $\pi.|\det(.)|^{(n-1)/2}$ is
$C$-algebraic.  Furthermore, for groups such as $\SL_2$ in which half
the sum of the positive roots is in $X^*(T)$, the notions of
$L$-algebraic and $C$-algebraic coincide.  On the other hand, as the
previous example of $\PGL_2/\Q$ attempted to illustrate, one does not
always have this luxury of being able to pass easily between
$L$-algebraic and $C$-algebraic representations for a given group
$G$. Furthermore a lot of naturally occurring representations are
$C$-algebraic: for example any cohomological automorphic
representation will always be $C$-algebraic (see Lemma
\ref{lem:cohomogicalimpliesCalgebraic} below) and, as the case of
$\PGL_2/\Q$ illustrated, there may be natural candidates for Galois
representations associated to these automorphic representations, but
they may not take values in the $L$-group of $G$! In fact, essentially all known
examples of Galois representations attached to automorphic
representations ultimately come from the cohomology of Shimura
varieties (although in some cases the constructions also use
congruence arguments), and this cohomology is naturally decomposed in
terms of cohomological automorphic representations. Much of the rest
of this paper is devoted to examining the relationship between
$C$-algebraic and $L$-algebraic in greater detail, and
defining a ``$C$-group'', which should conjecturally receive
the Galois representations attached to $C$-algebraic automorphic
representations.

\section{The case of tori.}\label{section:tori}\subsection{}

In this section we prove
conjectures~\ref{conj:LarithmeticLalgebraic},
~\ref{conj:CarithmeticCalgebraic}, \ref{conj:existence of Galois representations}
and~\ref{conj:existence of Galois representations - strong     version with crystalline etc} when $G$ is a torus
over a number field~$F$.
That we can do this should not be considered surprising. Indeed,
if $G=\GL_1$ then the
results have been known for almost 30 years, and the general
case can be reduced to the $\GL_1$ case via base change and
a $p$-adic version of results of Langlands on
the local and global Langlands correspondence for tori.
Unfortunately we have not been able to find a reference
which does what we want so we include some of the details here.

First note that if $G$ is a torus then half the sum of the positive
roots is zero, so the Satake isomorphism
preserves $\Q$-structures and hence the notions of $C$-arithmetic
and $L$-arithmetic coincide and we can use the phrase ``arithmetic''
to denote either of these notions. Furthermore, again
because half the sum of the positive
roots is zero, the notions of $C$-algebraic and $L$-algebraic
also coincide, and we can use the phrase ``algebraic'' to mean
either of these two notions (and in the case $G=\GL_1/F$ this
coincides with the classical definition, and with Weil's notion of
being of type $(A_0)$).

Recall that to give a torus $G/F$
is to give its character group, which (after choosing an $\overline{F}$)
is a finite free $\Z$-module
equipped with a continuous action of $\Gal(\overline{F}/F)$. Let
$K\subset\overline{F}$
denote a finite Galois extension of~$F$ which splits $G$; then this action
of $\Gal(\overline{F}/F)$ factors through $\Gal(K/F)$.
An automorphic representation of $G/F$
is just a continuous group homomorphism $G(F)\backslash G(\A_F)\to\C^\times$.

Let $BC$ denote the usual base change map from automorphic representations
of $G/F$ to automorphic representations of $G/K$, induced
by the norm map $N:G(\A_K)\to G(\A_F)$. 

\begin{lemma}\label{algBC} If $\pi$ is an automorphic representation
of $G/F$ then $\pi$ is algebraic if and only if $BC(\pi)$ is.
\end{lemma}

\begin{proof} This is a local statement, and if we translate
it over to a statement about representations of Weil groups
then it says that if $k$ is an algebraic closure of $\R$
then $r:W_{\R}\to{}^LG(\C)$ is algebraic iff its restriction
to $W_k$ is, which is clear because our definition of algebraicity
of $r$ only depended on the restriction of~$r$ to $W_k$.
\end{proof}

Now let $T$ denote a torus over a local field~$k$, and assume $T$
splits over an unramified extension of~$k$. The topological
group $T(k)$ has a unique maximal compact subgroup $U$. Let $\chi$
be a continuous group homomorphism $T(k)\to\C^\times$ with
$U$ in its kernel. Note that $U$ is hyperspecial and hence $\chi$
is unramified.

\begin{lemma} For $E$ a subfield of $\C$, the following are equivalent:

(i) $\chi$ is defined over $E$.

(ii) The Satake parameter of $\chi$ is defined over $E$.

(iii) The image of $\chi$ is contained in~$E^\times$.
\end{lemma}

\begin{proof} The Satake isomorphism in this situation
is simply the identity isomorphism
$\C[T(k)/U]=\C[T(k)/U]$, which induces the identity isomorphism
$\Q[T(k)/U]=\Q[T(k)/U]$, so (i) and (ii) are equivalent. The equivalence
of (i) and (iii) follows from the statement that $\chi:T(k)/U\to\C^\times$
is $E^\times$-valued if and only if the induced ring homomorphism
$\Q[T(k)/U]\to\C$ is $E$-valued.
\end{proof}

If $k_1/k$ is a finite extension of local fields and if $T/k$ is a torus
then we also use the notation $\BC$ to denote the map
$\Hom(T(k),\C^\times)\to\Hom(T(k_1),\C^\times)$ induced by the
norm map $N:T(k_1)\to T(k)$. Now suppose again that $T$ is an unramified
torus over $k$ and $\chi:T(k)\to\C^\times$ is an unramified character.

\begin{corollary} If $\chi$ is defined over $E$ and if $k_1/k$ is a finite
extension, then $\BC(\chi)$ is defined over $E$.
\end{corollary}
\begin{proof} The image of $\BC(\chi)$ is contained in the image of $\chi$.
\end{proof}

We now go back to the global situation. Let $\pi$ denote an automorphic
representation of $G/F$, with $G$ a torus, and let $\BC(\pi)$ denote
its base change to $G/K$, where $K$ is a finite Galois extension of $F$
which splits $G$.

\begin{corollary}\label{cor: base change of arithmetic is arrithmetic} If $\pi$ is arithmetic then $\BC(\pi)$ is arithmetic.
\end{corollary}

\begin{proof} Immediate from the previous corollary.
\end{proof}

\begin{thm}\label{thm: alg and arith the same for split tori} If $G$ is a split torus over a number field, then
the notions of arithmetic and algebraic coincide.
\end{thm}
\begin{proof} The fact that algebraic implies arithmetic is standard; the other
implication is Th\'eor\`eme~5.1
of~\cite{MR693328} (which uses a non-trivial result in transcendence
theory).
\end{proof}

\begin{corollary}\label{ar_implies_al}
If $G$ is a torus over a number field~$F$
and  $\pi$ is an automorphic representation of $G$, and 
$\pi$ is arithmetic, then $\pi$ is algebraic.
\end{corollary}
\begin{proof} If $\pi$ is arithmetic then its base change
to $K$ (a splitting field for~$G$)
is arithmetic (by Corollary \ref{cor: base change of arithmetic is arrithmetic}), and hence algebraic by the previous theorem.
Hence $\pi$ is algebraic by Lemma~\ref{algBC}.
\end{proof}

To show that algebraic automorphic representations for~$G$
are arithmetic, we give a re-interpretation of what it means
for an automorphic representation of a torus to
be algebraic; we are grateful to Ambrus P\'al for pointing out to us
that such a re-interpretation should exist. First some notation.
Let $F_\infty:=F\otimes_\Q\R$. Let
$\Sigma$ denote
  the set of embeddings $\sigma:F\into\Qbar$. Recall that because
we have fixed an embedding $\Qbar\to\C$, each $\sigma\in\Sigma$
can be regarded as an embedding $F\to\C$, and hence
induces maps $F_\infty\to\C$ and
$\sigma_\infty:G(F_\infty)\to G_{\sigma}(\C)$, where $G_\sigma$ is
the group over $\Qbar$ induced from $G$ via base extension via $\sigma$.

\begin{prop}\label{prop: characterising algebraic representations of
    tori in terms of their infinite part}
The representation
$\pi$ is algebraic if and only if for each $\sigma\in\Sigma$
there is an algebraic character
$\lambda_\sigma:G_\sigma\to\GL_1/\Qbar$ such that $\pi$
agrees with
$\prod_{\sigma\in\Sigma}\lambda_\sigma\circ\sigma_\infty$ on $G(F_\infty)^0$
(the identity component of the Lie group $G(F_\infty)$).
\end{prop}
\begin{proof} 
This statement is local at infinity, and can be checked by ``brute force'',
explicitly working out what the local Langlands correspondence for
tori over the reals and complexes is and noting that it is true in
every case.
\end{proof}

\begin{corollary}\label{cor: valued in a number field at infinity} If $\pi$ is an algebraic automorphic representation
of $G/F$, and if we write $\pi=\pi_f\times\pi_\infty$, with
$\pi_\infty:G(F_\infty)\to\C^\times$, then $\pi_\infty(G(F))$
is contained within a number field.
\end{corollary}

\begin{proof} 
By the preceding proposition we know that $\pi_\infty|_{G(F)}$
is the product of a character of order at most~2 by a continuous group
homomorphism $G(F)\to\C^\times$
which is the product of maps
$\phi_\sigma:G(F)\stackrel{\sigma}{\to}G_\sigma(\C)\to\C^*$ given by
composing an
algebraic character with an embedding $\sigma:F\into\C$. Hence it suffices
to prove that $\phi(G(F))$ is contained within a number field for such
a $\phi$. However both $T$ and $\mathbb{G}_m$ are defined over $F$, so
the character descends to some number field $L$, which we may assume
splits $T$ and contains the images of all embeddings $F\into\C$. But then $\phi(F)\subset\phi(L)\subset L^\times$, as required.
\end{proof}

\begin{thm} The notions of arithmetic and algebraic coincide for
automorphic representations of tori over number fields.
\end{thm}
\begin{proof} Let $G/F$ be a torus over a number field, and let $\pi$
be an automorphic representation of $G$. By Corollary~\ref{ar_implies_al}
we know that if $\pi$ is arithmetic then $\pi$ is algebraic, so we only have
to prove the converse. We will make repeated use of the trivial observation
(already used above) that if $X$ is a finite index subgroup of an
abelian group $Y$, then the image of a character of $Y$ is contained
in a number field if and only if the image of its restriction to $X$ is contained
in a (possibly smaller) number field.

Let $\pi=\otimes_v\pi_v$ be algebraic. If $K$ is a finite
Galois extension of $F$ splitting~$G$ then $BC_{K/F}(\pi)$ is algebraic
by Lemma~\ref{algBC} and hence arithmetic by Theorem \ref{thm: alg and arith the same for split tori}.
Hence there is a number field~$E$
and some finite set~$S_K$ of places of~$K$, containing
all the infinite places, such
that for $w\not\in S_K$, $\BC(\pi)_w$ is defined over~$E$, and
hence $\BC(\pi)_w$ has image in $E^\times$. By increasing $S_K$
if necessary, we can assume that $S_K$ is precisely the set of places
of~$K$ lying above a finite set~$S$ of places of~$F$.

Let $N:G(\A_K)\to G(\A_F)$ denote the norm map. Standard results from
global class field theory (see for example p.244 of~\cite{MR1610871}
for the crucial argument)
imply that $G(F)N(G(\A_K))$ is a closed and open subgroup of finite index
in $G(\A_F)$. Hence if $\A_F^S$ denotes the restricted product of
the completions of $F$ at places not in~$S$, and $\A_K^{S_K}$ denotes
the analogous product for $K$, then $G(F)N(G(\A_K^{S_K}))$ has
finite index in $G(\A_F^S)$. Let $\pi^S:G(\A_F^S)\to\C^\times$
denote the restriction of $\pi$ to $G(\A_F^S)$. Then
$\pi=\pi^S.\prod_{v\in S,v\nmid\infty}\pi_v.\pi_\infty$.
We know that $\pi$ is trivial on $G(F)$ (by definition) and that $\pi_{\infty}$ sends
$G(F)$ to a number field (by Corollary~\ref{cor: valued in a number field at infinity}).
We also know  that
$\pi_v(G(F_v))$ (and thus $\pi_v(G(F))$ is contained within
a number field for each finite place $v\in S$ (because $BC_{K/F}(\pi)$
is arithmetic, we know that $\pi_v(N(K_w))$ is contained in a number
field, where $w|v$ is a place of $K$, and $N(K_w)$ has finite index in
$F_v$). Hence $\pi^S(G(F))$
is contained within a number field. Then since $G(F)N(G(\A_K^{S_K}))$
has finite index in $G(\A_F^S)$, we deduce that $\pi(G(\A_F^S))$ is contained within a number field, and
hence $\pi$ is arithmetic, as required.
\end{proof}

What remains now is
to prove Conjecture \ref{conj:existence of Galois representations - strong
    version with crystalline etc} (which of course implies Conjecture
  \ref{conj:existence of Galois representations}). This follows
  straightforwardly from Langlands' proof of the Langlands
  correspondence for tori, and the usual method for associating Galois
  representations to algebraic representations of $\Gm$. Take $\pi$,
  $G$, $F$ and $K$ as above, and again let $\Sigma$ denote
the field embeddings $F\to\Qbar$, noting now that because of our fixed
embedding $\Qbar\subset\Qpbar$ we can also interpret each element of $\Sigma$
as a field embedding $F\to\Qpbar$.

The first step is to associate a
``$p$-adic automorphic representation''---a continuous group homomorphism
$\pi_p:G(F)\backslash G(\A_F)\to\Qpbar^\times$---to $\pi$, which we do by
mimicking the standard construction in the split case.
For $\sigma\in\Sigma$ recall that $\sigma_\infty$ is the
induced map $G(F_\infty)\to G_\sigma(\C)$.
By Proposition \ref{prop: characterising algebraic representations of
    tori in terms of their infinite part} there are characters
  $\lambda_\sigma\in X^*(G_\sigma)$ (regarded here as maps
$G_\sigma(\C)\to\C^\times$)
for each $\sigma\in \Sigma$ with the
  property that \[\pi|_{G(F_\infty)^0}=\prod_{\sigma\in
    \Sigma}\lambda_\sigma\circ \sigma_\infty.\]
The right hand side of the above equation can be regarded as a character
of $G(F_\infty)$ and hence as a character $\lambda_\infty$
of $G(\A_F)$, trivial at the finite places. Define
$\pi^{\alg}=\pi/\lambda_\infty$, a continuous group homomorphism
$G(\A_F)\to\C^\times$ trivial on $G(F_\infty)^0$ but typically
non-trivial on $G(F)$. However $\pi^{\alg}(G(F))$ is
contained within a number field by Corollary~\ref{cor: valued in a number field at infinity}, and it is now easy to check
that the image of $\pi^{\alg}$ is contained within $\Qbar$.
We now regard $\pi^{\alg}$ as taking values in $\Qpbar$ via our
fixed embedding $\Qbar\subset\Qpbar$.

Now, let
  $F_p=F\otimes_\Q\Q_p$, and note that every $\sigma:F\to\Qpbar$ in $\Sigma$
induces a map $F_p\to\Qpbar$ and hence a map
$\sigma_p:G(F_p)\to G_\sigma(\Qpbar)$.
Each $\lambda_\sigma$ can be regarded as
a map $G_\sigma(\Qpbar)\to\Qpbar^\times$, and hence the product
$$\lambda_p:=\prod_\sigma\lambda_\sigma\circ\sigma_p$$
is a continuous group homomorphism $G(F_p)\to\Qpbar^\times$
and can also be regarded as a continuous group homomorphism
$G(\A_F)\to\Qpbar^\times$, trivial at all places other than those above~$p$.
The crucial point, which is easy to check, is that the
product $\pi_p:=\pi^{\alg}\lambda_p$ is a continuous group
homomorphism $G(\A_F)\to\Qpbar^\times$ which is trivial on $G(F)$.

Now, in Theorem 2(b) of~\cite{MR1610871},
Langlands proves that there is a natural
surjection with finite kernel from the set of $\widehat{G}(\C)$-conjugacy
classes of continuous homomorphisms
from the Weil group $W_F$ to ${}^LG(\C)$ to the set of continuous
homomorphisms from $G(F)\backslash G(\A_F)$ to $\C^\times$ (that is,
the set of automorphic representations of $G/F$), compatible with the
local Langlands correspondence at every place. His proof starts by
establishing a natural surjection from the analogous sets with
the continuity conditions removed, and then checking that continuity
on one side is equivalent to continuity on the other. 
However, $\Qpbar\cong\C$ as abstract fields,
and one can check that the calculations on pages 243ff
make no use of any particular features of the topology of $\C^\times$,
and hence apply equally well to continuous homomorphisms $W_F\to{}^LG(\Qpbar)$
and continuous characters $G(F)\backslash
G(\A_F)\to\Qpbar^\times$. Thus $\pi_p$
gives a continuous homomorphism (or perhaps several, in which case we
simply choose one) \[r_\pi:W_F\to{}^LG(\Qpbar),\]
and by construction we see that for each finite place $v\nmid p$ at which
$\pi_v=(\pi_p)_v$ is unramified, $r_\pi|_{W_{F_v}}$ is
$\widehat{G}(\Qpbar)$-conjugate to $r_{\pi_v}$. Again, by
construction the composite of $r_{\pi}$ and the natural
projection ${}^LG(\Qpbar)\to\Gal(\overline{F}/F)$ is just the natural
surjection $W_F\to\Gal(\overline{F}/F)$.
\begin{lem}
  \label{lem: Galois rep of torus really does facto through the
    Galois group}The representation $r_\pi$ of $W_F$ factors
  through the natural surjection $W_F\to\Gal(\overline{F}/F)$.
\end{lem}
\begin{proof}
The kernel of the natural surjection  $W_F\to\Gal(\overline{F}/F)$ is
the connected component of the identity in $W_F$; but $r_{\pi}$
must vanish on this, because ${}^LG(\Qpbar)$ is totally disconnected.
\end{proof}
We let $\rho_\pi$ denote the representation of
$\Gal(\overline{F}/F)$ determined by $r_\pi$.
\begin{lem}
  The representation $\rho_\pi$ satisfies all the properties
  required in the statement of Conjecture \ref{conj:existence of Galois representations - strong
    version with crystalline etc}.
\end{lem}
\begin{proof}
  We need to check the claimed properties at places dividing $p$ and
  at real places. For the former, we must firstly check that
  $\rho_\pi$ is de Rham with the correct Hodge--Tate
  weights. However, it is sufficient to check this after restriction
  to any finite extension of $F$, and in particular we may choose an
  extension which splits $G$. The evident compatibility of the
  construction of $\rho_\pi$ with base change then easily
  reduces us to the split case, which is standard. Similarly, the
  property of being crystalline may be checked over any unramified
  extension, and if $\pi_v$ is unramified then by definition $G$
  splits over an unramified extension of $F_v$, and we may again
  reduce to the split case.

Suppose now that $v$ is a real place of $F$. Recall that the natural
surjection $W_{F_v}\to\Gal(\overline{F_v}/F_v)$ sends $j$ to complex
conjugation, so we need to
determine $r_\pi|_{W_{F_v}}(j)$. Let $\sigma_v:F\into\C$
denote the embedding corresponding to $v$ (it is unique because $v$
is a real place) and let $\lambda_v$ denote the character $\lambda_{\sigma_v}$
of $G_{\sigma_v}$. Let $\chi_v$ denote the map $G(F_v)\to\C^\times$ induced
by $\sigma_v$ and $\lambda_v$. Then $(\pi_p)_v=\pi_v/\chi_v$. Applying
local Langlands we see that the cohomology class in $H^1(W_{F_v},\widehat{G})$
associated to $(\pi_p)_v$ is the difference of those associated to $\pi_v$
and $\chi_v$ (because local Langlands is an isomorphism of groups in
this abelian setting). Furthermore, one can check (either by the construction
of the local Langlands correspondence for real tori in section~9.4
of~\cite{MR546608}, or an explicit case-by-case check) that the
cohomology class attached to $\chi_v$ is represented by a cocycle
which sends $j\in W_{F_v}$ to the element $\lambda_v(-1)$ of $\widehat{G}$
(where we now view $\lambda_v$ as a cocharacter $\C^\times\to\widehat{G}$).
Our assertion about $r_\pi|_{W_{F_v}}(j)$ now follows immediately from
an explicit calculation on cocycles.
\end{proof}
\section{Twisting and Gross' $\eta$.}\label{section:twisting}
\subsection{Algebraicity and arithmeticity under central extensions.}

\begin{defn}\label{defn:gmextension}
  We will call a central extension
$$1\to \Gm\to G'\to G \to 1 $$ of algebraic groups over $F$
 a \emph{$\Gm$-extension} of $G$.
\end{defn}
Note that one can (after making compatible choices of maximal compact
subgroups at infinity, and using the fact that $H^1(F,\Gm)=0$
by Hilbert 90) identify automorphic representations on $G$ with automorphic
representations on $G'$ which are trivial on $\Gm$.
We abuse notation slightly and speak about the $\Gm$-extension $G'\to G$
interchangeably with the $\Gm$-extension $1\to \Gm\to G'\to G\to 1$.

We now consider how our various notions of arithmeticity and algebraicity
behave under pulling back via a $\Gm$-extension. So say $G'\to G$
is a $\Gm$-extension. The induced map $G'(\A_F)\to G(\A_F)$
is a surjection,
and if $\pi$ is an automorphic representation of $G(\A_F)$ then
the induced representation~$\pi'$ of $G'(\A_F)$
is also an automorphic representation. Furthermore, we have the
following compatibilities between $\pi$ and $\pi'$.
\begin{lemma}\label{z-extension_preserves_stuff}
$\pi$ is $L$-algebraic (resp.\ $C$-algebraic, resp.\ $L$-arithmetic,
resp.\ $C$-arithmetic) if and only if $\pi'$ is.
\end{lemma}
\begin{proof}
Let us start with the $L$- and $C$-algebraicity assertions. These
assertions follow from purely local assertions at infinity: one
needs to check that if $k$ is an archimedean local field (a completion
of $F$ in the application), and if $\pi$ is a representation of $G(k)$,
with $\pi'$ the induced representation of $G'(k)$, then $\pi$
is $L$-algebraic (resp.\ $C$-algebraic) if and only if $\pi'$ is.
These statements can easily be checked using infinitesimal characters.
Indeed a straightforward calculation (using an explicit description of
the Harish-Chandra isomorphism) shows that if $T'$ is a maximal
torus in $G'$ over the complexes (where we base change $G'$ to
the complexes via the map $k\to\C$ induced
from $\sigma:\kbar\to\C$), and if the image of $T'$ in~$G$
is $T$ (a maximal torus
of $G$), and if $\lambda_\sigma$ and $\lambda'_\sigma$ are
the elements of $X^*(T)\otimes_\Z\C$ and $X^*(T')\otimes_\Z\C$
corresponding to $\pi$ and $\pi'$ as in \S~\ref{locallanglandsatinfinity},
then the natural map
$X^*(T)\otimes_\Z\C\to X^*(T')\otimes_\Z\C$ sends $\lambda_\sigma$
to $\lambda'_\sigma$ and both results follow easily.

It remains to prove the arithmeticity statements. Again these statements
follow from purely local assertions. Let $k$ denote a non-archimedean
local field (a non-archimedean completion of $F$ at which everything
is unramified in the application) and let $\pi$ be an irreducible smooth
representation of $G(k)$, with $\pi'$ the corresponding representation
of $G'(k)$. Assume $\pi'$ and $\pi$ (and hence $G'$ and $G$) are unramified.
The $C$-arithmeticity assertion of the Proposition follows from the
assertion that $\pi$ is defined over a subfield~$E$ of~$\C$ iff $\pi'$
is; this is however immediate from Lemma~\ref{definedoverE}.
The $L$-arithmeticity statement follows from the assertion that
the Satake parameter of $\pi$ is defined over~$E$ iff the Satake
parameter of $\pi'$ is defined over~$E$, which is then what remains to be proved.
So let $T'$ be the centralizer of a maximal split torus in $G'$
over $k$, and let $T$ be its image in $G$. Then $T'(k)\to T(k)$
is a surjection by Hilbert 90. As noted in the proof of
Lemma~\ref{Larith_well_defined}, Theorem~2.9 of~\cite{MR0579172} shows
that the $W_d$-orbit of complex
characters of $T(k)$ determined by the Satake isomorphism
applied to $\pi$ are precisely the characters for which $\pi$
occurs as a subquotient of the corresponding induced representations,
and the analogous assertion also holds for~$\pi'$. It follows
that the orbit of characters of $T'(k)$ corresponding to $\pi'$
is precisely the orbit induced from the characters of $T(k)$ via the surjection
$T'(k)\to T(k)$. This implies that the Satake parameter of $\pi'$
(thought of as a character of $\Q[X_*(T'_d)]^{W_d}$)
is induced from the Satake parameter of $\pi$ via a map
between the corresponding unramified Hecke algebras which is
in fact the obvious map, and our assertion now follows easily.
\end{proof}

\subsection{Twisting elements.}
We now explain the relationship between $L$-algebraic and
$C$-algebraic automorphic representations for a connected reductive
group~$G$ over a number field~$F$. In particular, we examine
the general question of when $L$-algebraic representations can be
twisted to $C$-algebraic representations, following an idea of Gross
(see \cite{MR1729443}). We show that in general it is always possible
to replace $G$ by a $\Gm$-extension for which this twisting is possible,
and in this way one can formulate general conjectures about the
association of Galois representations to $C$-algebraic (and
in particular cohomological by Lemma
\ref{lem:cohomogicalimpliesCalgebraic} below) automorphic
representations. As usual, let $X^*$ denote the character group in
the based root datum for~$G$, with its Galois action. Let us stress
that we always equip $X^*$ with the Galois action coming from the
construction used to define the $L$-group (which might well not be
the same as the ``usual'' Galois action on $X^*(T)$ induced by the
Galois action on $T(\overline{F})$, if $T$ is a maximal
torus in $G$ which happens to be defined over $F$).

\begin{defn}\label{twistingelement}
	We say that an element $\theta\in X^{*}$ is a \emph{twisting element} if $\theta$ is $\Gal(\overline{F}/F)$-stable and $\langle\theta,\alpha^{\vee}\rangle=1\in\Z$ for all simple coroots $\alpha^{\vee}$.
\end{defn}
For some groups $G$ there are no twisting elements; for example, if
$G=\PGL_{2}$. On the other hand, if $G$ is semi-simple and simply-connected
then half the sum of the positive roots is a twisting element.
Another case where twisting elements exist are groups~$G$
that are split and have simply-connected derived subgroup, for example $G=\GL_n$, although in this case half the sum of the
positive roots might not be in $X^*$.

If~$Q$ is the quotient of~$G$ by its derived subgroup, then
$X^*(Q)\subseteq X^*$ and the arguments in section~II.1.18 of~\cite{MR2015057}
show that $X^*(Q)=(X^*)^{W}$, where~$W$ is the Weyl group of $G_{\overline{F}}$.
Furthermore, $X^*(Q)$ is $\Gal(\overline{F}/F)$-stable, and the induced
action of $\Gal(\overline{F}/F)$ on $X^*(Q)$ is precisely the usual
action, induced by the Galois action on $Q(\overline{F})$. 

Now let $\delta$ denote half the sum of the positive roots of $G$. 
If $\delta\in X^*$ then $\delta$ is a twisting element; but in
general we only have $\delta\in\frac{1}{2}X^*$.
Let $S'$ denote the maximal split torus quotient of $G$, so
that $$X^{*}(S')=(X^{*})^{W,\Gal(\overline{F}/F)}.$$
Then if $\theta$ is a twisting element, we see
that $$\theta-\delta\in\frac{1}{2} X^{*}(S').$$ Thus we have a
character $|\cdot |^{\theta-\delta}$ of $G(F)\backslash G(\A_{F})$,
defined as the composite $$\xymatrix{G(\A_{F})\ar[r]&
  S'(\A_{F})\ar[r]^{2(\theta-\delta)}&
  \A_{F}^{\times}\ar[r]^{|\cdot|}&\R_{>0}\ar[r]^{x\mapsto\sqrt{x}}&\R_{>0}} $$
The main motivation behind the notion of twisting elements is the
following two propositions.
\begin{prop}\label{prop:twistinglalgtocalg}If $\theta$ is a twisting element, then an automorphic representation $\pi$ is $C$-algebraic if and only if $\pi\otimes|\cdot |^{\theta-\delta}$ is $L$-algebraic.
\end{prop}
\begin{proof} This is a consequence of condition 10.3(2)
  of~\cite{MR546608}, although formally one has to
  ``reverse-engineer'' the construction of (using the notation of
  \S10.2 of loc.\ cit.) $\alpha\mapsto\pi_\alpha$.  We sketch the
  argument using the notation there. The question is local at each
  infinite place, so let $k$ denote a completion of~$F$ at an infinite
  place. Let $\widetilde{G}$ be the central extension of~$G$
described in \S10.2 of~\cite{MR546608}. The character
  $|\cdot|^{\theta-\delta}$ induces a character of $G(k)$ and of
  $\widetilde{G}(k)$. If $Q$ denotes the maximal torus quotient of
  $\widetilde{G}$ then this character can be extended to a character
  of $Q(k)$. The associated element of $H^1(W_k,\widehat{Q})$ is the
  image of an element $\alpha\in H^1(W_k,Z_L)$, with $Z_L$ the centre
  of $\widehat{G}$, and one checks easily that the character
  $\pi_\alpha$ of $G(k)$ in \S10.2 of~\cite{MR546608} coincides with
  $|\cdot|^{\theta-\delta}$.  If $T_0$ is a maximal torus in
  $\widetilde{G}$, and $T$ is its image in $G$, then the restriction
  of $\alpha$ to $\overline{k}^\times$ is a $Z_L$-valued character
  which, when considered as a $\widehat{T}_0$-valued character of
  $\overline{k}^\times$, has image in $\widehat{T}$ and which (via an
  easy diagram chase) coincides with the restriction to
  $\overline{k}^\times$ of the cohomology class associated via local
  Langlands to the restriction of $|\cdot|^{\theta-\delta}$ to
  $T(k)$. Hence $a$ on $\overline{k}^\times$ is the composite of the
  norm map down to $\R_{>0}$, the square root map, and the cocharacter
  of $\widehat{T}$ associated to $2(\theta-\delta)$.  Twisting $\pi$ by
  $|\cdot|^{\theta-\delta}$ corresponds to twisting $r_{\pi}$ by $a$ by
  10.3(2) of~\cite{MR546608} and the result follows easily.
\end{proof}

\begin{prop}\label{prop:twistinglarithtocarith}If $\theta$ is a
  twisting element, then $\pi$ is $C$-arithmetic if and only if
  $\pi\otimes|\cdot |^{\theta-\delta}$ is $L$-arithmetic.
\end{prop}
\begin{proof} Again this is a local issue: by
Definitions \ref{defn:L-arithmetic}
and \ref{defn:C-arithmetic} it suffices to
check that if $k$ (a completion of~$F$)
is a non-archimedean local field, if $\chi$ denotes the restriction
of $|\cdot|^{\theta-\delta}$ to $G(k)$ and if $\pi$ is an unramified representation
of $G(k)$, then $\pi$ is defined over a subfield $E$ of $\C$ iff
$\pi\otimes\chi$ has Satake parameter defined over
the same subfield $E$. This is relatively easy to check: we sketch
the details (using the notation of section~\ref{satakeparams}).
Let $T$ be a maximal torus of $G/k$, with maximal
compact subgroup ${}^oT$.
Then $\chi$ induces an
automorphism $i$ of $H_{\C}(T(k),{}^oT)$ sending $[{}^oTt{}^oT]$
to $\chi(t)[{}^oTt{}^oT]$, and $i$ commutes with the action of
the Weyl group $W_d$ and hence induces an automorphism~$i$ 
of the $W_d$-invariants of this complex Hecke algebra.
If $m_\pi$ is the complex character of $H(T(k),{}^oT)^{W_d}$ associated
to $\pi$ and $m_{\pi\otimes\chi}$ is the character associated to $\pi\otimes\chi$
then one checks easily that $m_{\pi\otimes\chi}=m_\pi\circ i$. The other
observation we need is that if $K$ is a hyperspecial maximal compact
subgroup of $G(k)$ and if $S$ denotes the Satake isomorphism
$S:H_{\C}(G(k),K)\to H_{\C}(T(k),{}^oT)^{W_d}$ then
$i\circ S$ maps $H_{\Q}(G(k),K)$ into $H_{\Q}(T(k),{}^oT)$ (this
follows immediately from formula (19) of section 4.2
of~\cite{cartier:corvallis}), and
hence into $H_{\Q}(T(k),{}^oT)^{W_d}$, and an
injection between $\Q$-vector spaces which becomes an isomorphism
after tensoring with $\C$ must itself be an isomorphism. Hence
$i\circ S:H_{\Q}(G(k),K)\cong H_{\Q}(T(k),{}^oT)^{W_d}$ and
now composing with $m_{\pi}$ the result follows easily.
\end{proof}
Thus for groups with a twisting element, our $L$-notions and $C$-notions
can be twisted into each other.

\subsection{Adjoining a twisting element.}

If~$G$ has a twisting element then we have just seen that one
can, by twisting, pass between our $L$- and $C$- notions.
What can one do when $G$ has no twisting element (for example
if $G=\PGL_2$)? In this case we will show that~$G$ has a $\Gm$-extension
$\widetilde{G}$ which \emph{does} have a twisting element. Here we use
some ideas that we learnt from reading a 2007 letter from Deligne to Serre;
we thank Dick Gross for drawing our attention to this letter.
We remark that previous versions of this manuscript contained a slightly
messier construction involving a two-step process, reducing 
first via a $z$-extension to the the case where $G$ had simply-connected derived
subgroup and then making another extension from there (this procedure
of reducing to the case of a simply-connected derived subgroup seems
to be often used in the literature but it turned out not to be necessary
in this case).

\begin{prop}\label{prop:newgoodcover}

(a) There is a $\Gm$-extension
$$1\to\Gm\to \widetilde{G}\to G\to 1 $$
such that $\widetilde{G}$ has a canonical twisting element~$\theta$.

(b) If $G$ has a twisting element then $\widetilde{G}\cong G\times\Gm$.
More generally, there is a natural bijection between the set
of splittings of $1\to\Gm\to\widetilde{G}\to G\to 1$ and the
set of twisting elements of~$G$ (and in particular, the sequence
splits over~$F$ if and only if~$G$ has a twisting element).
\end{prop}
\begin{proof}

(a) Let $G^{\ad}$ denote the quotient of~$G$ by its
centre and let $G^{\sc}$ denote the simply-connected cover
of $G^{\ad}$. Over $\overline{F}$ we can choose compatible
(unnamed) Borels and tori $T^{\sc}$ and $T^{\ad}$ in $G^{\sc}$ and $G^{\ad}$,
and hence define the notion of a positive root in $X^*(T^{\sc})$
and $X^*(T^{\ad})$.

Now $G^{\sc}$ is simply-connected, and hence half the
sum of the positive roots for $G^{\sc}$ is a character~$\eta$ of $T^{\sc}$
and hence induces a character~$\gamma$ of $Z:=\ker(G^{\sc}\to G^{\ad})$. This
character $\gamma$ is independent of the notion of positivity because
all Borels in~$G^{\sc}_{\overline{F}}$ are conjugate. It takes
values in $\mu_2$ because $\eta^2$, the sum of the positive
roots, is a character of $T^{\ad}$ and is hence trivial on~$Z$.
Furthermore $\gamma$ is independent of the choice of $T^{\sc}$,
and defined over~$F$.

Pushing the diagram
$$0\to Z\to G^{\sc}\to G^{\ad}\to 0$$
out along $\gamma:Z\to\Gm$ gives us an extension
$$0\to\Gm\to G^1\to G^{\ad}\to 0$$
of groups over~$F$
and now pulling back along $G\to G^{\ad}$ gives us an extension
$$0\to\Gm\to\widetilde{G}\to G\to 0.$$
The group $\widetilde{G}$ is the extension we seek. One can
define it ``all in one go'' as a subquotient of $G\times G^{\sc}\times\Gm$:
it is the elements $(g,h,k)$ such that the images of $g$ and $h$
in $G^{\ad}$ are equal, modulo the image of the finite group~$Z$
under the map sending $z$ to $(1,z,\gamma(z))$. 

If $T^1$ and $\widetilde{T}$ are maximal tori in $G^1$ and $\widetilde{G}$
then the character groups of these tori fit into the following
Galois-equivariant commutative diagram
\[\xymatrix{
0\ar[r]&X^*(T^{\ad})\ar[r]&X^*(T^{\sc})\ar[r]&X^*(Z)\ar[r]& 0\\
0\ar[r]&X^*(T^{\ad})\ar[r]\ar@{=}[u]\ar[d]&X^*(T^1)\ar[r]\ar[u]\ar[d]&\Z\ar[r]\ar[u]^{\gamma^*}\ar@{=}[d]&0\\
0\ar[r]&X^*(T)\ar[r]&X^*(\widetilde{T})\ar[r]&\Z\ar[r]&0}
\]
where $X^*(T^1)$ can be thought
of as a pullback---the subgroup of $X^*(T^{\sc})\oplus\Z$ consisting
of elements whose images in $X^*(Z)$ coincide, and $X^*(\widetilde{T})$
can be thought of as a pushforward---the quotient of $X^*(T^1)\oplus X^*(T)$
by $X^*(T^{\ad})$ embedded anti-diagonally.

We define $\theta$ thus: there is an element $\theta^1$ of $X^*(T^1)$
whose image in $X^*(T^{\sc})$ is $\eta$, and whose image in~$\Z$ is~1. We let $\theta$ be
the image of $\theta^1$ in $X^*(\widetilde{T})$.

We claim that $\theta$ is a twisting element for $\widetilde{G}$,
and this will suffice to
prove part (a). Note that the snake lemma implies that $X^*(T^1)$
injects into $X^*(\widetilde{T})$ and hence it suffices to show that
$\theta^1$ is a twisting element for $G^1$. It is a standard fact that
$\eta$ pairs to one with each simple coroot, and it follows
immediately that $\theta^1$ pairs to one with each simple
coroot. Furthermore, $\eta$ is a Galois-stable element of
$X^*(T^{\sc})$, and Galois acts trivially on $\Z$, from which one can
deduce that Galois acts trivially on $\theta$, and so $\theta$ is
indeed a twisting element.

(b) Recall that we have a short exact sequence
$$0\to X^*(T)\to X^*(\widetilde{T})\to\Z\to 0.$$
If $t$ is a twisting element for~$G$, then $t\in X^*(T)$
can be regarded as an element of $X^*(\widetilde{T})$ whose
image in~$\Z$ is~zero. Recall that $\theta\in X^*(\widetilde{T})$
from (a) is a twisting element for~$\widetilde{G}$, whose
image in~$\Z$ is~1. Now consider the difference $\chi:=\theta-t$.
This is an element in~$X^*(\widetilde{T})$ which is Galois
stable and pairs to zero with each simple coroot. Hence if~$\widetilde{Q}$
denotes the maximal split torus quotient of~$\widetilde{G}$, then $\chi$
is an element of the subgroup~$X^*(\widetilde{Q})$ of~$\widetilde{T}$.
In particular, $\chi$ induces a map $\widetilde{G}\to\Gm$ defined
over~$F$, and the composite $\Gm\to\widetilde{G}\to\Gm$ is the
identity map. It follows easily that the induced
map $\widetilde{G}\to G\times\Gm$ induced by $\chi$ and
the canonical map $\widetilde{G}\to G$ is an isomorphism.

This same argument shows that there is a bijection between the
splitting elements for~$G$ and the splittings of the short
exact sequence. To give a splitting of the exact sequence
is to give a map $\widetilde{G}\to\Gm$ such that the induced map
$\Gm\to\widetilde{G}\to\Gm$ is the identity. If~$\chi:\widetilde{G}\to\Gm$ is
such a map then $\chi$ gives an element
of~$X^*(\widetilde{T})$ which is Galois stable, pairs to zero with
each simple coroot, and whose image in~$\Z$ is~1. Conversely
any such element gives a splitting. Now one checks easily that
$\chi\in X^*(\widetilde{T})$ has these properties then
$\theta-\chi$ ($\theta$ our fixed twisting element for~$\widetilde{G}$
coming from part (a)) has image zero in~$\Z$, so can be regarded
as an element of~$X^*(T)$ that can easily be checked to be a twisting
element for~$G$, and conversely for any twisting element~$t\in X^*(T)$,
we have already seen that $\theta-t$ gives a splitting.
\end{proof}

The importance of $\widetilde{G}$ is that if $\pi$ is a $C$-algebraic
representation for $G$, we can pull it back to $\widetilde{G}$
and then twist as in Proposition~\ref{prop:twistinglalgtocalg}
to get an $L$-algebraic representation, for which we predict
the existence of a Galois representation---but into the $L$-group
of $\widetilde{G}$ rather than the $L$-group of~$G$. The construction
of $\widetilde{G}$ is completely canonical, and this motivates
the following definition:
\begin{defn} The \emph{$C$-group of~$G$} is defined
to be the $L$-group of $\widetilde{G}$. We will denote this group by
${}^CG$.
\end{defn}
The $C$-group is ``as functorial as the $L$-group is'',
and naturally receives the Galois representations conjecturally
(and in some cases provably) attached to $C$-algebraic automorphic
representations for~$G$. One can think of the $C$-group as doing,
in a canonical way, the job of unravelling all the square roots that
appear in the Satake isomorphism normalised a la Langlands when
applied to a cohomological representation (for we shall see later
on in Lemma~\ref{lem:cohomogicalimpliesCalgebraic} that cohomological
representations are $C$-algebraic).
Note that the dimension of the $C$-group
is one more than the dimension of the $L$-group, but this extra
degree of freedom is cancelled out by the fact that the inclusion
$\Gm\to\widetilde{G}$ gives rise to a map $d:{}^CG\to\Gm$
and we will see later on that our conjectures imply that
if $\pi$ is $C$-algebraic for $G$ then the associated
Galois representation, when composed with the map
$d:{}^CG(\Qpbar)\to\Gm(\Qpbar)$, is always the cyclotomic character.

It would be nice to see in a concrete manner how the $C$-group
is related to the $L$-group and we achieve this in the next
proposition. Let $\chi$ denote the sum of the positive roots
for~$G$ (after choosing some Borel and torus); then $\chi$
can be thought of as a cocharacter of a maximal torus $\widehat{T}$
of~$\widehat{G}$, that is, a map $\Gm\to\widehat{T}$. Set
$e=\chi(-1)$.
\begin{prop}\label{kernelorder2} The element~$e$ is a central element of $\widehat{G}$
and is independent of all choices. There is a canonical surjection
$$\widehat{G}\times\Gm\to\widehat{\widetilde{G}}$$
with kernel central and of order~2, generated by $(e,-1)$.
This surjection is equivariant for the action of Galois
used to define the $L$-groups of~$G$ and~$\widetilde{G}$.
\end{prop}
\begin{proof}
We have a canonical map $c:\widetilde{G}\to G$; let us beef
this up to an isogeny $\widetilde{G}\to G\times\Gm$ with
kernel of order~2. To do this we need to construct an appropriate
map $\xi:\widetilde{G}\to\Gm$, which we do thus: recall that
the intermediate group $G^1$ used in the construction of $\widetilde{G}$
was a push-out of $G^{\sc}\times\Gm$ along~$Z$; we define a map
$G^{\sc}\times\Gm\to\Gm$ by sending $(g,\lambda)$ to $\lambda^2$;
the image of~$Z$ in~$\Gm$ has order at most~2, and thus this
induces a map $G^1\to\Gm$ and hence a map $\xi:\widetilde{G}\to\Gm$.
The restriction of $\xi$ to the subgroup $\Gm$ of~$\widetilde{G}$
is the map $\lambda\mapsto\lambda^2$, and hence the kernel
of the induced map $(c,\xi):\widetilde{G}\to G\times\Gm$
has order~2 and hence this map is an isogeny for dimension reasons.
This isogeny induces a dual isogeny
$$\widehat{G}\times\Gm\to\widehat{\widetilde{G}}$$
with kernel of order~2, and so to check that the kernel
is generated by $(e,-1)$ it suffices to prove that $(e,-1)$
is in the kernel. But one easily checks that $(\chi,1)$ maps to $2\theta$ under
the natural map \[X_*(\widehat{T})\oplus\Z\to
X_*(\widehat{\widetilde{T}})\]
induced by the isogeny, and the result follows immediately by
evaluating this cocharacter at $-1$. Furthermore $(c,\xi)$
is defined over~$F$ so the dual isogeny commutes with the Galois
action used to define the $L$-groups.
\end{proof}

The interested reader can use the preceding Proposition to
compute examples of $C$-groups.
For example the $C$-group
of $(\GL_n)_F$ is isomorphic to the group
$(\GL_n\times\GL_1)_{\Qbar}\times\Gal(\overline{F}/F)$,
as $e=(-1)^{n-1}$ and the surjection $\GL_n\times\GL_1\to\GL_n\times\GL_1$
sending $(g,\mu)$ to $(g\mu^{n-1},\mu^2)$ has kernel $(e,-1)$.
As another example, the $C$-group of $(\PGL_2)_F$ is isomorphic
to $(GL_2)_{\Qbar}\times\Gal(\overline{F}/F)$ because the
kernel of the natural map $\SL_2\times\GL_1\to\GL_2$ is $(e,-1)$.
Note however that the $C$-group of $(\PGL_n)_F$ is not in
general $(\GL_n)_{\Qbar}\times\Gal(\overline{F}/F)$;
its connected component is $(\SL_n\times\GL_1)/\langle(-1)^{n-1},-1\rangle$,
which is isomorphic to $\SL_n\times\GL_1$ if $n$ is odd,
and to a central extension of $\GL_n$ by a cyclic group
of order $n/2$ if~$n$ is even.

One can use the above constructions and Conjecture~\ref{conj:existence of Galois
  representations} to formulate a conjecture associating ${}^CG$-valued Galois
representations to $C$-algebraic automorphic representations (
and hence, by Lemma
\ref{lem:cohomogicalimpliesCalgebraic} below, to cohomological
automorphic representations)
for an arbitrary connected reductive group~$G$ over a number field.
One uses Lemma~\ref{z-extension_preserves_stuff} to pull the $C$-algebraic
representation back to a $C$-algebraic representation on~$\widetilde{G}$,
twists using $\theta$ and the recipe in the statement of
Proposition~\ref{prop:twistinglalgtocalg} to get
an $L$-algebraic representation on~$\widetilde{G}$,
and then uses Conjecture~\ref{conj:existence of Galois
  representations} on this bigger group. The map denoted~$|.|^{\theta-\delta}$
in the aforementioned proposition can be checked to be the map $\widetilde{G}(\A_F)\to\R_{>0}$
sending $g$ to $|\xi(g)|^{1/2}$, with $\xi:\widetilde{G}\to\Gm$ the map defined
in the proof of Proposition~\ref{kernelorder2}. Note finally that the
map $\Gm\to\widetilde{G}$ induces a map $d:{}^CG\to\Gm$.

Unravelling, and in particular using Remark~\ref{unramlltwist}
with the~$G$ there being our~$\widetilde{G}$,
we see that Conjecture~\ref{conj:existence of Galois
representations} implies the following conjecture.
\begin{conj}\label{conj:existence of Galois representations - C-group
    version}If $\pi$ is $C$-algebraic, then there is a finite
subset $S$ of the places of $F$, containing all infinite places, all places dividing $p$, and
all places where $\pi$ is ramified, and a continuous Galois representation
$\rho_\pi=\rho_{\pi,\iota}:\Gal(\overline{F}/F)\to\CG(\Qpbar)$, which satisfies
\begin{itemize}
	\item  The composite of $\rho_\pi$ and the natural
          projection $\CG(\Qpbar)\to \Gal(\overline{F}/F)$ is the
          identity map.
        \item The composite of $\rho_\pi$ and $d$ is the
cyclotomic character.
	\item If $v\notin S$, then $\rho_\pi|_{W_{F_{v}}}$ is $\Gtildedual(\Qpbar)$-conjugate to
the representation sending $w\in W_{F_v}$ to $\iota(r_{\pi_{v}}(w)\hat{\xi}(|w|^{1/2}))$, where $\hat{\xi}$
is the map $\C^\times\to\Gtildedual(\C)$ dual to $\xi$, and where the norm on the Weil group
sends a geometric Frobenius to the reciprocal of the size of the residue field.
	\item If $v$ is a finite place dividing $p$ then
          $\rho_\pi|_{\Gal(\overline{F_v}/F_v)}$ is de
          Rham, and the Hodge--Tate cocharacter of this representation
can be explicitly read off from $\pi$ via the recipe of Remark
\ref{rem: recipe for HT weights in the C-case.} below.
          \item If $v$ is a real place, let $c_{v}\in G_{F}$ denote a
            complex conjugation at $v$, let $\lambda_{\sigma_v}$ and
$\lambda_{\tau_v}\in X^*(T)\otimes\C$ be the elements associated
to $\pi_v$ in \S\ref{locallanglandsatinfinity}, and
let $\tilde\lambda_{\sigma_v}$ and $\tilde\lambda_{\tau_v}$ denote
their images in $X^*(\widetilde{T})\otimes\C$. Then
$\tilde\lambda_{\sigma_v}+\frac{1}{2}\xi$, $\tilde\lambda_{\tau_v}+\frac{1}{2}\xi\in X^*(\widetilde{T})=X_*(\widehat{\widetilde{T}})$ and $\rho_{\pi,\iota}(c_{v})$
            is $\Gdual(\Qpbar)$-conjugate to the element
            \[\iota(\alpha_v)=\iota\left(\left(\tilde{\lambda}_{\sigma_v}+\frac{1}{2}\xi\right)(i)\left(\tilde{\lambda}_{\tau_v}+\frac{1}{2}\xi\right)(i)r_{\pi_v}(j)\right)\]
            associated to the twist of the lift of $\pi$, as in Lemma~\ref{cxconj}.
\end{itemize}	
\end{conj}
\begin{rem}\label{rem: recipe for HT weights in the C-case.} The recipe
for the Hodge--Tate cocharacter in the conjectures above is as follows.
As in Remark~\ref{rem: recipe for HT cocharacter}, any $j:F\to\Qbar$
gives rise to an infinite place~$w$ of~$F$ equipped with a fixed
map $F_w\to\C$, which we use to identify~$\C$ with an algebraic
closure of~$F_w$. The identity $\sigma:\C\to\C$ and the representation
$\pi_w$ give rise to $\lambda_\sigma$ as in
Remark~\ref{rem: recipe for HT cocharacter},
except that this time $\lambda_\sigma\in
(X^*(T)\otimes_\Z\frac{1}{2}\Z)/W$. We thus obtain an element
$\tilde{\lambda}_\sigma\in (X^*(\Ttilde)\otimes_\Z\frac{1}{2}\Z)/W$. Then
$\tilde{\lambda}_\sigma+\frac{1}{2}\xi\in X^*(\Ttilde)/W$, and our conjecture
is that this element $\tilde{\lambda}_\sigma+\frac{1}{2}\xi$ is the Hodge--Tate cocharacter
associated to the embedding $F_v\to\Qpbar$.
\end{rem}

We end this section by showing that the results in it imply the equivalence of Conjectures~\ref{conj:LarithmeticLalgebraic} and~\ref{conj:CarithmeticCalgebraic} (made for all groups simultaneously).
\begin{prop}\label{conjs_are_equiv} Let $G$ be a connected reductive
  group over a number field. If
  Conjecture~\ref{conj:LarithmeticLalgebraic} is true for $\widetilde{G}$ then
  Conjecture~\ref{conj:CarithmeticCalgebraic} is true
  for~$G$. Similarly if Conjecture~\ref{conj:CarithmeticCalgebraic} is
  true for $\widetilde{G}$ then
  Conjecture~\ref{conj:LarithmeticLalgebraic} is true for~$G$.
\end{prop}
\begin{proof} We prove the first assertion; the second one is similar.
Say $G$ is connected and reductive, and $\pi$ is $C$-arithmetic
(resp.\ $C$-algebraic). Let $\pi'$ be the pullback of $\pi$ to $\widetilde{G}$.
Then $\pi'$ is $C$-arithmetic (resp.\ $C$-algebraic) by
Lemma~\ref{z-extension_preserves_stuff}.
By Proposition~\ref{prop:twistinglarithtocarith} (resp.
Proposition~\ref{prop:twistinglalgtocalg})
$\pi'\otimes|.|^{\theta-\delta}$ is $L$-arithmetic (resp.\ $L$-algebraic).
Applying Conjecture~\ref{conj:LarithmeticLalgebraic} to
$\widetilde{G}$ we deduce that
$\pi'\otimes|.|^{\theta-\delta}$ is $L$-algebraic (resp.\ $L$-arithmetic).
Running the argument backwards now shows us that $\pi$ is $C$-algebraic
(resp.\ $C$-arithmetic).
\end{proof}
\section{Functoriality.}\label{sec:functoriality}
\subsection{}
Suppose that $G$, $G'$ are two connected reductive groups over $F$, and that we have an \emph{algebraic}
$L$-group homomorphism $$r:\LG\to\LG',$$ i.e. a homomorphism of algebraic
groups over $\Qbar$ which respects the projections to $\Gal(\overline{F}/F)$
(recall that we are using the Galois group rather than the Weil group
when forming the $L$-group). Assume that $G'$ is quasi-split over $F$. Then we have the following weak version of Langlands' functoriality conjecture (note that we are only demanding compatibility with the local correspondence at a subset of the unramified places, and at infinity).

\begin{conj}\label{conj:functoriality}If $\pi$ is an automorphic representation of $G$, then there is an automorphic representation $\pi'$ of $G'$, called a \emph{functorial transfer} of $\pi$, such that \begin{itemize}
	\item  For all infinite places $v$, and for all finite places  $v$ at which $\pi$ and $G'$ are unramified, $r_{\pi'_{v}}$ is $\Gdual'(\C)$-conjugate to $r\circ r_{\pi_{v}}$.
\end{itemize}
\end{conj}
A trivial consequence of the definitions is
\begin{lemma}\label{lem:transferoflalgislalg}If $\pi$ is $L$-algebraic, then any functorial transfer of $\pi$ is also $L$-algebraic.
\end{lemma}
We also have the only slightly less trivial
\begin{lemma}\label{lem:transferoflarithislarith}If $\pi$ is $L$-arithmetic, then any functorial transfer of $\pi$ is also $L$-arithmetic.
\end{lemma}
\begin{proof} This result follows from a purely local assertion.
If $v$ is a finite place
where $G$ and $G'$ are unramified and if~$k$ is the completion of~$F$
at~$v$ then the morphism~$r$ of $L$-groups induces a morphism
$\widehat{G}\to\widehat{G}'$ which commutes with the action of the Frobenius
at~$v$. If $T_d$ (resp.\ $T'_d$) denotes a maximal $k$-split torus in $G/k$
(resp. $G'/k$) with centralizer $T$ (resp.\ $T'$) then the map
$\widehat{G}\to\widehat{G}'$ induces a map
$\widehat{T}\to\widehat{T}'$ (well-defined up to restricted Weyl group actions)
which commutes with the action of Frobenius,
and hence maps $\widehat{T}_d\to\widehat{T}'_d$ and
$X^*(\widehat{T}'_d)=X_*(T'_d)\to X^*(\widehat{T}_d)=X_*(T_d)$. Now looking at the
explicit definition of the Satake isomorphism in Proposition~6.7
of~\cite{MR546608} we see, after unravelling, that the map
$\Q[X_*(T'_d)]^{W'_d}\to\Q[X_*(T_d)]^{W_d}$ induced from $X_*(T'_d)\to X_*(T_d)$
above has the property
that, after tensoring up to $\C$ and taking spectra, it sends
the point in $\Spec(\C[X_*(T_d)]^{W_d})$ corresponding to $r_{\pi_v}$
to the point in $\Spec(\C[X_*(T'_d)]^{W'_d})$ corresponding
to $r\circ r_{\pi_v}$. We now deduce that if the Satake parameter
of $\pi_v$ is defined over a subfield $E$ of $\C$ then so is the
Satake parameter of $\pi'_v$ (because the homomorphism
$\Q[X_*(T'_d)]^{W'_d}\to\C$ corresponding to $\pi'_v$
factors through $\Q[X_*(T_d)]^{W_d}$ and hence through $E$) and
the result follows.
\end{proof}
In addition,
\begin{prop}\label{prop:galoisrepsforfunctorialtransfer}If Conjecture \ref{conj:existence of Galois representations} holds for $\pi$, then it holds for any functorial transfer of $\pi$.
\end{prop}
\begin{proof}With notation as above, one easily checks that $\rho_{\pi',\iota}:=r\circ\rho_{\pi,\iota}$ satisfies all the conditions of Conjecture \ref{conj:existence of Galois representations}.\end{proof}

Note that functoriality relies on things normalised in Langlands' canonical
way; the natural analogues of the results above in the $C$-algebraic
and $C$-arithmetic cases are not true in general, because
a morphism of algebraic groups does not send half the sum of the positive
roots to half the sum of the positive roots in general.

\section{Reality checks.}\label{reality checks}\subsection{}By Proposition 2 of \cite{langlands:notion} any automorphic representation $\pi$ on $G$ is a subquotient of an induction $\Ind_{P(\A_{F})}^{G(\A_{F})}\sigma$, where $P$ is a parabolic subgroup of $G$ with Levi quotient $M$, and $\sigma$ is a cuspidal representation of $M$. If $\pi'$ is another automorphic subquotient of $\Ind_{P(\A_{F})}^{G(\A_{F})}\sigma$, then $\pi_{v}$ and $\pi'_{v}$ are equal for all but finitely many places, so $\pi$ is $C$-arithmetic (respectively $L$-arithmetic) if and only if $\pi'$ is $C$-arithmetic (respectively $L$-arithmetic). The following lemma shows that $\pi$ is $C$-algebraic (respectively $L$-algebraic) if and only if $\pi'$ is $C$-algebraic (respectively $L$-algebraic).

\begin{lemma}\label{lem:compatibilitywithinductions}Suppose that $\pi$ and $\pi'$ are subquotients of a common induction $\Ind_{P(\A_{F})}^{G(\A_{F})}\sigma$. Then $\pi$ and $\pi'$ have the same infinitesimal character.
\end{lemma}
\begin{proof}This is immediate from the calculation of the infinitesimal character of an induction---see for example Proposition 8.22 of \cite{KnappOverviewExamples}.
\end{proof}
Furthermore, we can check the compatibility of Conjecture
\ref{conj:existence of Galois representations} for $\pi$ and $\pi'$
(note that we cannot check the compatibility for Conjecture \ref{conj:existence of Galois representations - strong
    version with crystalline etc} because $\pi$ and $\pi'$ may be
  ramified at different places).

\begin{prop}\label{prop:Galois reps for constituents of inductions}Suppose that $\pi$ and $\pi'$ are subquotients of a common induction $\Ind_{P(\A_{F})}^{G(\A_{F})}\sigma$. Suppose that $\pi$ is $L$-algebraic. If Conjecture \ref{conj:existence of Galois representations} is valid for $\pi$ then it is valid for $\pi'$.
\end{prop}
\begin{proof}Suppose that Conjecture \ref{conj:existence of Galois
    representations} is valid for $\pi$. We wish to show that
  $\rho_{\pi',\iota}:=\rho_{\pi,\iota}$ satisfies all the
  conditions in Conjecture \ref{conj:existence of Galois
    representations}. Since for all but finitely many places $\pi_{v}$
  and $\pi'_{v}$ are unramified and isomorphic, the first two
  conditions are
  certainly satisfied. The third condition is satisfied by Lemma
  \ref{lem:compatibilitywithinductions}.  It remains to check that if
  $v$ is a real place, then (with obvious notation)
  $\lambda_\sigma(i)\lambda_\tau(i)r_{\pi_v}(j)$ and
  $\lambda'_\sigma(i)\lambda'_\tau(i)r_{\pi'_v}(j)$ are
  $\Gdual(\C)$-conjugate. As explained to us by David Vogan, it
  follows from the results of \cite{MR1162533} (specifically from
  Theorem 1.24 and Proposition 6.16) that
  $\lambda_\sigma(-1)r_{\pi_v}(j)$ and $\lambda'_\sigma(-1)r_{\pi'_v}(j)$ are
  $\Gdual(\C)$-conjugate. It is easy to check that these are $\widehat{G}(\C)$-conjugate
  to $\lambda_\sigma(i)\lambda_\tau(i)r_{\pi_v}(j)$ and
  $\lambda'_\sigma(i)\lambda'_\tau(i)r_{\pi'_v}(j)$ respectively (one
conjugates by $r_{\pi_v}(e^{-i\pi/4})$, $r_{\pi'_v}(e^{-i\pi/4})$), as required.
\end{proof}
More generally, if $\pi$, $\pi'$ are nearly equivalent (that is,
$\pi_{v}$ and $\pi'_{v}$ are isomorphic for all but finitely many
$v$), then $\pi$ is $C$-arithmetic (respectively $L$-arithmetic) if
and only if $\pi'$ is $C$-arithmetic (respectively $L$-arithmetic). We
would like to be able to prove as above that $\pi$ and $\pi'$ have the
same infinitesimal character, and we would like to obtain the analogue
of Proposition \ref{prop:Galois reps for constituents of
  inductions}. Unfortunately, these seem in general to be beyond
the reach of current techniques. However, we can prove these results
for $\GL_{n}$, and we can then deduce them for general groups under
the assumption of functoriality.

\begin{prop}\label{prop:nearequivforGLn}If $G=\GL_{n}$ and $\pi$,
  $\pi'$ are nearly equivalent, then $\pi$, $\pi'$ have the same
  infinitesimal character.  Suppose further that $\pi$ is
  $L$-algebraic. If Conjecture \ref{conj:existence of Galois
    representations} is valid for $\pi$ then it is valid for $\pi'$.
\end{prop}
\begin{proof}By the strong multiplicity one theorem for isobaric
  representations (Theorem 4.4 of \cite{MR623137}), $\pi$, $\pi'$ are
  both subquotients of a common induction
  $\Ind_{P(\A_{F})}^{G(\A_{F})}\sigma$. The result follows from Lemma
  \ref{lem:compatibilitywithinductions} and Proposition
  \ref{prop:Galois reps for constituents of inductions}.
\end{proof}
\begin{prop}
  Let $G$ be arbitrary. Assume Conjecture \ref{conj:functoriality}. If $\pi$ and $\pi'$
  are nearly equivalent automorphic representations of~$G$, then for any infinite place $v$, $\pi_{v}$
  and $\pi'_{v}$ have the same infinitesimal characters. Suppose
  further that $\pi$ is $L$-algebraic. If Conjecture
  \ref{conj:existence of Galois representations} is valid for $\pi$
  then it is valid for $\pi'$.
\end{prop}
\begin{proof}For each infinite place $v$ choose an injection $\Fbar\to\Fvbar$.
This gives us a natural injection ${}^{L}G_{v}\to{}^{L}G$, where $G_{v}$ is the
  base change of $G$ to $F_{v}$. Since $r_{\pi_{v}}|_{\C^{\times}}$ is
  valued in ${}^{L}G_{v}$, as is
  $c_{v}:=\lambda_\sigma(i)\lambda_\tau(i)r_{\pi_v}(j)$ if $v$
is real, we see that their
  $\Gdual(\C)$-conjugacy classes in ${}^{L}G$ are determined by their
  $\Gdual(\C)$-conjugacy classes in ${}^{L}G_{v}$.
	
  Now, the $\Gdual(\C)$-conjugacy classes of semisimple elements of
  ${}^{L}G_{v}(\C)$ are determined by the knowledge of the conjugacy
  classes of their images under all representations of
  ${}^{L}G_{v}(\C)$. To see this, note that since the formation of the
  $L$-group is independent of the choice of inner form, it suffices to
  check this in the case where $G_{v}$ is quasi-split; but the result
  then follows immediately from Proposition 6.7 of \cite{MR546608}.
	
	Let $r:\LG\to\GL_{n}\times\Gal(\overline{F}/F)$ be a homomorphism of $L$-groups. Then by Conjecture \ref{conj:functoriality}, there are automorphic representations $\Pi$, $\Pi'$ on $\GL_{n}$ which are functorial transfers of $\pi$, $\pi'$ respectively. By Proposition \ref{prop:nearequivforGLn}, $\Pi$ and $\Pi'$ have the same infinitesimal characters. Thus for each infinite place $v$, $r_{\Pi_{v}}|_{\C^{\times}}$ and $r_{\Pi'_{v}}|_{\C^{\times}}$ are conjugate, i.e. $r\circ r_{\pi,\iota}|_{\C^{\times}}$ and $r\circ r_{\pi',\iota}|_{\C^{\times}}$ are conjugate. Since this is true for all $r$, we see that $r_{\pi,\iota}|_{\C^{\times}}$ and $r_{\pi',\iota}|_{\C^{\times}}$ are conjugate, whence $\pi$ and $\pi'$ have the same infinitesimal character.
	
As in the proof of Proposition \ref{prop:Galois reps for constituents of inductions}, it remains to check that if $v$ is a real place of $F$, then
$c_{v}:=\lambda_\sigma(i)\lambda_\tau(i)r_{\pi_v}(j)$ and
$c'_{v}:=\lambda'_\sigma(i)\lambda'_\tau(i)r_{\pi'_v}(j)$ are $\Gdual(\C)$-conjugate. By a similar argument to that used in the first half of this proof, we see that if $r:\LG\to\GL_{n}\times\Gal(\overline{F}/F)$ is a homomorphism of $L$-groups, then $r(c_{v})$ and $r(c'_{v})$ are conjugate in $\GL_{n}(\C)$.
Furthermore, $c_{v}$ and $c'_{v}$ are both semisimple.
Thus $c_{v}$ and $c'_{v}$ are $\Gdual(\C)$-conjugate.
\end{proof}
\subsection{Cohomological representations.}\label{sec:cohomological}
Cohomological automorphic representations provide a good testing ground for our conjectures. It follows easily (see below) that any cohomological representation is $C$-algebraic, and one can often show that they are $C$-arithmetic, too
(it would not surprise us if Shimura variety experts could prove they
were always $C$-arithmetic with relative ease).
In the case $G=\GL_{n}$ these arguments are due to Clozel, who also shows that for $\GL_{n}$ any regular $C$-algebraic representation is cohomological after possibly twisting by a quadratic character (see Lemme 3.14 of \cite{MR1044819}).

Let $v$ be an infinite place of $F$, and let $K_{v}$ be the fixed choice of a maximal compact subgroup of $G(F_{v})$ used in the definition of automorphic forms on $G$. Let $\mathfrak{g}_{v}$ be the complexification of the Lie algebra of $G(F_{v})$. Recall that $\pi_{v}$ may be thought of as a $(\mathfrak{g}_{v},K_{v})$-module, with underlying $\C$-vector space $V_{v}$, say. 

\begin{defn}
	We say that $\pi_{v}$ is \emph{cohomological} if there is an algebraic complex representation $U$ of $G(F_{v})$ and a non-negative integer $i$ such that $$H^{i}(\mathfrak{g}_{v},K_{v};U\otimes V_{v})\neq 0.$$ We say that $\pi$ is cohomological if $\pi_{v}$ is cohomological for all archimedean places $v$. 
\end{defn}

\begin{lemma}\label{lem:cohomogicalimpliesCalgebraic}If $\pi$ is cohomological, then it is $C$-algebraic.
	\end{lemma}
\begin{proof}By Corollary 4.2 of \cite{MR1721403}, if $\pi$ is cohomological then for each archimedean place $v$ there is a continuous finite-dimensional representation $U_{v}$ of $G(F_{v})$ such that $\pi_{v}$ and $U_{v}$ have the same infinitesimal characters. The result then follows from Lemma \ref{lem:infcharoffinitedimensional} below.\end{proof}
\begin{lemma}\label{lem:infcharoffinitedimensional}If $v$ is an archimedean place of $F$ and $U$ is a continuous finite dimensional representation of $G(F_{v})$ with infinitesimal character $\chi_{v}$, identified with an element of $X^{*}(T)\otimes_\Z\C$  as in section \ref{locallanglandsatinfinity}, then $\chi_{v}-\delta\in X^{*}(T)$.
	
\end{lemma}
\begin{proof}This follows almost at once from the definition of the Harish-Chandra isomorphism; see for example (5.43) in \cite{MR1920389}.\end{proof}
We note that a cuspidal cohomological unitary automorphic representation~$\pi$
is also $C$-arithmetic, at least when~$\pi_\infty$ is cohomological
for the trivial representation; for we can restrict scalars down to $\Q$ and
then follow the argument in \S2.3 of~\cite{MR1265563}. This argument
presumably works in some greater generality.

\section{Relationship with theorems/conjectures in the
  literature.}\label{comparison to gross and clozel}

\subsection{}In \cite{MR1044819}, Clozel makes a number of conjectures about certain $C$-algebraic automorphic representations for $\GL_{n}$. We now examine the compatibility of these conjectures with those of this paper. Clozel calls an automorphic representation of $\GL_{n}$ \emph{algebraic} if it is $C$-algebraic and isobaric; his principal reason for restricting to isobaric representations is that he wishes to use the language of Tannakian categories.

Let $\pi=\otimes'\pi_{v}$ be an algebraic (in Clozel's sense) representation of $\GL_{n}$ over $F$. Then Clozel conjectures (see conjectures 3.7 and 4.5 of \cite{MR1044819}) that

\begin{conj}\label{conj:Clozel}Let $\pi_{f}=\otimes'_{{v\nmid\infty}}\pi_{v}$. Then there is a number
field $E\subset\C$ such that $\pi_f$ is defined over $E$  (that is, such that $\pi_f\otimes_{\C,\sigma}\C\cong\pi_f$
for all automorphisms $\sigma$ of $\C$ which fix $E$ pointwise). In addition, Conjecture \ref{conj:existence of Galois representations - strong version with crystalline etc} holds for $\pi\otimes|\cdot|^{(n-1)/2}$.
\end{conj}
(In fact, Clozel conjectures much more than this---he conjectures that there is a motive whose local $L$-factors agree with those of $\pi\otimes|\cdot|^{(n-1)/2}$ at all finite places; the required Galois representation is then obtained as the $p$-adic realisation of this motive.)

By Proposition \ref{prop:nearequivforGLn} we see, since any automorphic representation of $\GL_{n}$ is nearly equivalent to an isobaric one, that Conjecture \ref{conj:Clozel} implies Conjecture \ref{conj:existence of Galois representations} for $\GL_{n}$, and in fact an examination of the proof shows that it implies Conjecture \ref{conj:existence of Galois representations - strong version with crystalline etc}. We claim that it also implies that $\pi$ is $C$-arithmetic; in fact, this follows at once from Proposition 3.1(iii) of \cite{MR1044819}. Thus for $\GL_{n}$ our conjectures follow from those of Clozel.

The reason that our conjectures are weaker than Clozel's conjectures
is that for groups other than $\GL_{n}$ we do not have as good an
understanding of the local Langlands correspondence---for a general
group~$G$ we cannot even \emph{formulate} a version of
Conjecture~\ref{conj:existence of Galois representations - strong
    version with crystalline etc} which includes behaviour at
the bad places without also formulating a precise local Langlands conjecture
in full generality. Even for $GL_n$ Clozel had to be careful, restricting
to isobaric representations in order not to make a conjecture which
was trivially false, and such phenomena would also show up in the general
case, and are typically even less well-understood here.

\subsection{}
In \cite{MR1729443} Gross presents a conjecture which assigns a Galois
representation to an automorphic representation on a group $G$ with
the property that any arithmetic subgroup is finite (in fact
Gross gives six conditions equivalent to this in Proposition 1.4 of
\cite{MR1729443}). We now discuss the relationship of
this conjecture to our conjectures. Gross' assumptions
imply that $G$ splits over a CM field $L$, and he also
assumes that~$G$ has a twisting element $\eta$ in the sense of
Definition \ref{twistingelement}. In fact, Gross has informed us that
one should in addition assume that the group $G$ is semisimple and
simply connected, so we make this assumption from now on. This
assumption in fact implies that one can take $\eta=\delta$ in the below, but
for those who want to be more optimistic than Gross we have kept
the two notations distinct in the below.

Let $V$ be an absolutely irreducible representation of $G$ over $\Q$
with trivial central character. Let $S$ be a finite set of primes of
size at least 2, containing all primes at which $G$ is ramified. For
each $\ell\notin S$ we let $K_{\ell}$ be a hyperspecial maximal compact
subgroup of $G(\Q_{\ell})$, and for each $\ell\in S$ we let $K_{\ell}$ be
an Iwahori subgroup of $G(\Q_{\ell})$. Let $K$ be
the product of the $K_\ell$. Then
$M(V,K)$ is the space of algebraic modular forms given by

\begin{align*}
M(V,K):=\{
f:G(\A_{\Q})/
(G(\R)_{+}&\times K)\to V:\\
&f(\gamma g)=\gamma f(g)\text{ for all }\gamma\in G(\Q), g\in G(\A_{\Q})
\} .
\end{align*}

Let $H_{S}$ be the unramified Hecke algebra---the restricted tensor
product of the
unramified Hecke algebras $H_{\ell}$ for each $\ell\notin S$. Let $H_{K}$ be
the full Hecke algebra, the tensor product of $H_{S}$ and the
Iwahori Hecke algebras $H_{\ell}$ at places $\ell$ in $S$. Let $A$ be
$H_{K}\otimes\Q[\pi_{0}(G(\R))]$. This acts on $M(V,K)$ (see section 6
of \cite{MR1729443}), and we let $N$ be a simple $A$-submodule of
$M(V,K)$. We assume that $N$ gives the Steinberg character on $H_{\ell}$
for all $\ell\in S$ (see section 12 of \cite{MR1729443}), and if $V$ is
trivial and $\prod_{l\in S}G(\Q_{l})$ is compact, we exclude the case
that $N$ is trivial.

By Proposition 12.3 of \cite{MR1729443}, $\End_{A}(N)$ is a CM field,
and by (7.4) of \cite{MR1729443}, $\pi_{0}(G(\R))$ acts on $N$ through
a character

$$\phi_{\infty}:\pi_{0}(G(\R))\to \{\pm 1\}\subset E^{\times}. $$

By Proposition 8.5 of \cite{MR1729443}, the simple submodules of
$N\otimes \C$ may be identified (compatibly with the actions of $H_K$)
with irreducible automorphic representations
$\pi=\pi_{f}\otimes\pi_{\infty}$ with $\pi_{\infty}\isoto V\otimes\C$,
and $\pi_{l}$ Steinberg for all $l\in S$. For all $l\notin S$, the
unramified local Langlands correspondence (i.e. the Satake
isomorphism) identifies the character of $T_{l}$ on $N$ with a
homomorphism $r_{N,l}:W_{\Q_{l}}\to\LG(\C)$, and $\pi_{l}$ corresponds
to this parameter under the local Langlands correspondence. Fix such a representation $\pi$.

Gross then makes the following conjecture (see Conjecture 17.2 as well
as (15.3) and (16.8) of \cite{MR1729443}) (note that while Gross
normalises his Weil groups so that an arithmetic Frobenius
element corresponds to a uniformiser, he also normalises his Satake
isomorphisms so that they are constructed with arithmetic Frobenii,
so the comments of Remark \ref{swapping geometric and arithmetic frobenius
  via an involution} apply):

\begin{conj}\label{conj:grossgaloisrepsconjecture}If $p$ is a prime,
  and $\iota:\C\isoto\Qpbar$, then there is a continuous Galois representation $$\rho_{N,\iota}:\Gal(\Qbar/\Q)\to\LG(\Qpbar) $$ satisfying \begin{itemize}
	\item If $l\notin S$, then $\rho_{N,\iota}|_{W_{\Q_{l}}}$ is $\Gdual(\Qpbar)$-conjugate to $\iota(r_{N,l})\otimes |\cdot|^{\eta-\delta}$.
	\item If $s_{\infty}$ is a complex conjugation in
          $\Gal(\Qbar/\Q)$, then $\rho_{N,\iota}(s_{\infty})$ is
          $\Gdual(\Qpbar)$-conjugate to
          $(\iota(\eta(-1)\phi_{\infty}(-1)),s_\infty)$.
\end{itemize}
	
\end{conj}
This conjecture follows from Conjecture \ref{conj:existence of Galois
  representations - strong version with crystalline etc}. Indeed, the
representation $\pi$ is $C$-algebraic, so by Proposition
\ref{prop:twistinglalgtocalg} $\pi\otimes |\cdot|^{\eta-\delta}$ is
$L$-algebraic. Applying Conjecture \ref{conj:existence of Galois
  representations - strong version with crystalline etc} gives
everything in Conjecture \ref{conj:grossgaloisrepsconjecture} (for the
description of complex conjugation, see \cite{grossodd}).

Note that Gross in fact conjectures something slightly stronger; he
shows that $\pi$ is $C$-algebraic, and in fact that $\pi$ is defined
over $E$, and conjectures that for any place $\lambda|p$ of $E$ there
is a natural Galois representation $\rho_{N,\lambda}:\Gal(\Qbar/\Q)\to
\LG(E_\lambda)$. As Gross has explained to us, this rationality
conjecture should follow
from the hypothesis that $\pi$ is Steinberg at two places, together
with local-global compatibility for the Galois representations at these places.

\subsection{}\label{section:examples}
We now discuss an example drawn from \cite{cht}. Let $F$
be a totally real field, and let $E$ be a quadratic totally imaginary
extension of $F$. Let $G$ be an $n$-dimensional unitary group over $F$
which splits over $E$, and which is compact (that is, isomorphic to
$U(n)$) at all infinite places. Then the dual
group of $G$ is $GL_{n}$, and if we let $\Gal(E/F)=\{1,c\}$, then the
$L$-group of $G$ is given
by $$\LG=\GL_{n}\rtimes\Gal(\overline{F}/F)$$
where $\Gal(\overline{F}/F)$ acts on $\GL_n$ via its
projection to $\Gal(E/F)=\{1,c\}$, with
$$x^{c}:=\Phi_{n}x^{-t}\Phi_{n}^{-1} $$ where
$\Phi_{n}$ is an anti-diagonal matrix with alternating entries $1$,
$-1$. Note that $\Phi_{n}^{-1}=\Phi_{n}^t=(-1)^{n-1}\Phi_{n}$.

By Proposition \ref{kernelorder2},
we
have $${}^{C}G=\left((\GL_{n}\times\Gm)/((-1)^{n-1},-1)\right)\rtimes\Gal(\overline{F}/F) $$ with
$\Gal(\overline{F}/F)$ acting
by $$(g,\mu)^{c}=(\Phi_{n}g^{-t}\Phi_{n}^{-1},\mu). $$

In Section 1 of \cite{cht} there is a definition of a group
$\mathcal{G}_{n}$. This group is  a semidirect product
$(\GL_{n}\times\Gm)\rtimes\Gal(\overline{F}/F)$, but with $\Gal(\overline{F}/F)$ acting
by $$(g,\mu)^{c}=(\mu g^{-t},\mu). $$ There is a
morphism $j:{}^CG\to\mathcal{G}_{n}$ defined thus.
For $(g,\mu)\in\GL_n\times\Gm$ and $\gamma\in\Gal(\overline{F}/E)$
we set
$$j((g,\mu)\times\gamma)=(g\mu^{1-n},\mu^{2(1-n)})\times\gamma$$
(and note that $j(((-1)^{n-1},-1)\times\gamma)=(1\times\gamma)$.
If $\tilde{c}\in\Gal(\overline{F}/F)$ has image $c\in\Gal(E/F)$
then we set
$$j(1\times\tilde{c})=(\Phi_n,(-1)^{n-1})\times\tilde{c}.$$
It is easily checked these determine a unique homomorphism
$j:{}^CG\to\mathcal{G}_n$. For $n>1$, $j$ is an isogeny
with kernel of order $n-1$.
Since $G$ is compact at infinity, any automorphic representation $\pi$
of $G(\A_F)$ is cohomological and thus by Lemma
\ref{lem:cohomogicalimpliesCalgebraic} is $C$-algebraic. Conjecture
\ref{conj:existence of Galois representations - C-group version}
predicts the existence of a Galois
representation\[\rho_\pi:\Gal(\overline{F}/F)\to{}^CG(\Qpbar)\]with
the property that $d\circ\rho_\pi$ is the cyclotomic character
$\varepsilon$ (with
$d$ as in the discussion preceding Conjecture \ref{conj:existence of
  Galois representations - C-group version}). One checks
easily that the composite
\[j\circ\rho_\pi:\Gal(\overline{F}/F)\to\mathcal{G}_n(\Qpbar)\] has
multiplier $\varepsilon^{1-n}$ (the multiplier of a representation
into $\mathcal{G}_n$ is its projection onto the $\Gm$ factor).

Now, under certain mild hypotheses on $G$ and $\pi$, we note that a
Galois representation $r'_\pi$ satisfying the properties imposed on $j\circ
r_\pi$ by Conjecture \ref{conj:existence of Galois representations -
  C-group version} is proved to exist in \cite{cht} and
\cite{twugk}. Specifically, everything except the form of complex
conjugation follows from Proposition 3.4.4 of \cite{cht} (although see
also Theorems 4.4.2 and Theorems 4.4.3 of \cite{cht} for related
results on $\GL_n$ whose notation may be easier to compare to the
notation used in this paper), and the form of complex conjugation
follows from Theorem 4.1 of \cite{twugk}.  (Note when comparing the
unramified places that by definition the local Langlands
correspondence $\rec(\pi_v)$ used in \cite{cht} is our $r_{\pi_v}$.)

Conversely, the constructions of \cite{cht}, when they
apply, actually imply our Conjecture
\ref{conj:existence of Galois representations - C-group version} for
$\pi$ (up to Frobenius semisimplification at unramified finite
places). To see this, note there is a morphism
$j':\mathcal{G}_n\times\GL_1\to{}^C G$ such that $j'\circ(j\times d)$
is the identity; concretely, $j'$ is defined on the identity component
by \[((h,\mu),\lambda)\mapsto
(h\lambda^{(n-1)/2},\lambda^{1/2})\](which is well-defined independent
of the choice of square root by the definition of~${}^CG$). Then we
may set $r_\pi=j'\circ (r'_\pi\times\varepsilon)$,
which is easily checked to have the required properties.

\bibliographystyle{amsalpha} 

\bibliography{buzzgee}

\end{document}